\documentclass[12pt]{amsart}

\usepackage[margin=1in]{geometry}

\usepackage{amsmath,amscd,amssymb,amsfonts,latexsym,wasysym, mathrsfs, mathtools,hhline,color}
\usepackage[all, cmtip]{xy}
\usepackage{csquotes}
\usepackage{url}
\usepackage{comment}

\definecolor{hot}{RGB}{65,105,225}

\usepackage[pagebackref=true,colorlinks=true, linkcolor=hot ,  citecolor=hot, urlcolor=hot]{hyperref}
\usepackage{ textcomp }
\usepackage{ tipa }
\usepackage{graphicx,enumerate}

\theoremstyle{plain}
\newtheorem{theorem}{Theorem}[section]
\newtheorem{prop}[theorem]{Proposition}

\newtheorem{lm}[theorem]{Lemma}

\newtheorem{cor}[theorem]{Corollary}
\newtheorem{conj}[theorem]{Conjecture}

\newtheorem{thrm}[theorem]{Theorem}

\theoremstyle{definition}

\newtheorem{defn}[theorem]{Definition}

\newtheorem{rmk}[theorem]{Remark}

\newtheorem{assumption}[theorem]{Assumption}
\newtheorem{ex}[theorem]{Example}
\newtheorem*{ex*}{Example}

\def\be{\begin{equation}}
\def\ee{\end{equation}}

\def\bt{\begin{thrm}}
\def\et{\end{thrm}}

\def\bc{\begin{cor}}
\def\ec{\end{cor}}

\def\br{\begin{rmk}}
\def\er{\end{rmk}}

\def\bp{\begin{prop}}
\def\ep{\end{prop}}

\def\bl{\begin{lm}}
\def\el{\end{lm}}

\def\bex{\begin{ex}}
\def\eex{\end{ex}}

\def\bd{\begin{defn}}
\def\ed{\end{defn}}

\newcommand\fp{{\mathfrak p}}
\newcommand\fm{{\mathfrak m}}

\newcommand\sH{{\mathcal H}}

\newcommand\sP{{\mathcal P}}

\newcommand\sF{{\mathcal F}}

\newcommand\sM{{\mathcal M}}

\newcommand\sV{{\mathcal V}}
\newcommand\sL{\mathcal{L}}

\newcommand{\Fp}{\mathbb{F}_p}
\newcommand\kk{{\mathbb{K}}}

\newcommand\zz{{\mathbb{Z}}}

\newcommand\cc{{\mathbb{C}}}

\DeclareMathOperator{\codim}{codim}              
                  
\DeclareMathOperator{\homo}{Hom}
\DeclareMathOperator{\rhomo}{RHom}

\DeclareMathOperator{\spec}{Spec}

\DeclareMathOperator{\rank}{Rank}

\DeclareMathOperator{\Perv}{Perv}
\DeclareMathOperator{\Alb}{Alb}
\DeclareMathOperator{\alb}{alb}

\DeclareMathOperator{\depth}{depth}

\DeclareMathOperator{\Hom}{Hom}

\DeclareMathOperator{\im}{Im}
\DeclareMathOperator{\Char}{Char}

\def\bC{\mathbb{C}}
\def\cM{\mathcal{M}}

\def\cH{\mathcal{H}}

\def\lra{\longrightarrow}

\def\bZ{\mathbb{Z}}

\newcommand{\ubul}{{\,\begin{picture}(-1,1)(-1,-3)\circle*{2}\end{picture}\ }}

\title[Perverse sheaves on semi-abelian varieties]{Perverse sheaves on semi-abelian varieties}

\author{Yongqiang Liu}
\address{Y. Liu: The Institute of Geometry and Physics, University of Science and Technology of China, 96 Jinzhai Road, Hefei 230026 P.R. China} 
\email{liuyq@ustc.edu.cn}
\author{Laurentiu Maxim}
\address{L. Maxim: Department of Mathematics,         University of Wisconsin-Madison,  480 Lincoln Drive, Madison WI 53706-1388, USA.}
\email {maxim@math.wisc.edu}
\author{Botong Wang}
\address{B. Wang: Department of Mathematics,         University of Wisconsin-Madison,  480 Lincoln Drive, Madison WI 53706-1388, USA.}
\email {wang@math.wisc.edu}
\date{\today}



\begin{document}

\maketitle

\begin{abstract}  
We give a complete (global) characterization of $\bC$-perverse sheaves on semi-abelian varieties in terms of their cohomology jump loci.
Our results generalize Schnell's work on perverse sheaves on complex abelian varieties,  as well as Gabber-Loeser's results   
on perverse sheaves on complex affine tori. 
We apply our results to the study of cohomology jump loci of smooth quasi-projective varieties, to the topology of the Albanese map, and in the context of homological duality properties of complex algebraic varieties.
\end{abstract}



\section{Introduction}\label{intro}

Perverse sheaves are fundamental objects at the crossroads of topology, algebraic geometry, analysis and differential equations, with important applications in number theory, algebra and representation theory.
They provide an essential tool for understanding the geometry and topology of complex algebraic varieties. 
For instance, the decomposition theorem \cite{BBD}, a far-reaching generalization of the Hard Lefschetz theorem of Hodge theory 
with a wealth of topological applications, requires the use of perverse sheaves. Furthermore, perverse sheaves are an integral part of Saito's theory of mixed Hodge module  \cite{Sa0, Sa1}. Perverse sheaves have also seen spectacular applications in representation theory, such as the proof of the Kazhdan-Lusztig conjecture, the proof of the geometrization of the Satake isomorphism,  or the 
proof of the fundamental lemma in the Langlands program (e.g., see  
\cite{CM} for a beautiful survey).  A proof of the Weil conjectures using perverse sheaves was given in \cite{KiW}.

However, despite their fundamental importance, 
perverse sheaves remain rather mysterious objects. In his 1983 ICM lecture, MacPherson \cite{MP} stated the following:
\begin{displayquote}
{\it The category of perverse sheaves is important because of its applications. It would be interesting to understand its structure more directly.}
\end{displayquote}
Alternative descriptions of perverse sheaves have since been obtained in various contexts, e.g., by MacPherson-Vilonen \cite{MV} in terms of zig-zags, by Gelfand-MacPherson-Vilonen \cite{GMV} by using quivers, etc. 

Perverse sheaves on complex affine tori have been studied by Gabber-Loeser \cite{GL} via the Mellin transformation, whereas perverse sheaves on complex abelian varieties have been completely characterized by Schnell \cite{Sch} by properties of their cohomology jump loci. 
The works of Gabber-Loeser \cite{GL} and Schnell \cite{Sch} are the departure point for this paper. Our main results give a complete (global) characterization of $\bC$-perverse sheaves on semi-abelian varieties in terms of their cohomology jump loci, generalizing Schnell's work and complementing Gabber-Loeser's results. 

\medskip

Let $X$ be a smooth connected complex quasi-projective variety. The {\it character variety} $\Char(X)$ is the connected component of $\Hom(\pi_1(X),\bC^*)$ containing the identity. It is isomorphic to
$(\bC^*)^{b_1(X)}$, and it  can be identified with the maximal spectrum $\spec \bC[H_{1,f}(X,\bZ)]$ of the $\bC$-group ring of the free part of $H_1(X,\bZ)$.   
Each character $\rho \in \Char(X)$ corresponds to a unique rank-one $\bC$-local system $L_{\rho}$ on $X$.
The {\it cohomology jump loci of a constructible complex} $\sF \in D^b_c(X,\bC)$ on $X$ are defined as:
 \be\label{jlc} \sV^i(X,\sF): = \{\rho \in \Char(X) \mid H^i(X, \sF \otimes_\bC L_{\rho}) \neq 0\}.\ee
These are generalizations of the cohomology jump loci $\sV^i(X):=\sV^i(X,\bC_X)$ of $X$, which correspond to the constant sheaf $\bC_X$, and which are homotopy invariants of $X$.

It was recently shown in \cite{BW17} that the irreducible components of the cohomology jump loci of bounded constructible complexes on any smooth complex algebraic variety $X$ are {\it linear} subvarieties. In particular, each $\sV^i(X,\sF)$ is a finite union of translated subtori of the character variety $\Char(X)$. This fact imposes strong constraints on the topology of $X$.

By the classical {Albanese map} construction  (e.g., see \cite{Iit}), cohomology jump loci of a smooth complex quasi-projective variety are realized as cohomology jump loci of constructible complexes of sheaves (or, if the Albanese map is proper, of perverse sheaves) on a semi-abelian variety. 
This partly motivates our study of cohomology jump loci of constructible complexes, with a view towards a complete characterization of perverse sheaves on complex  
semi-abelian varieties.
Besides providing new obstructions on the cohomology jump loci (hence also on the homotopy type) of smooth complex quasi-projective varieties, such a characterization has other important topological applications, such as finiteness properties of Alexander-type invariants (see, e.g., \cite{LMW2}), or 
for the study of homological duality properties of complex algebraic varieties.

\subsection{Main results}
A complex {\it semi-abelian variety} is a complex algebraic group $G$ which is an extension
$$1 \to T \to G \to A \to 1,$$
where $A$ is an abelian variety of complex dimension $g$ and $T\cong(\bC^*)^m$ is an algebraic affine torus of complex dimension $m$.
We set $\Gamma_G:=\bC[\pi_1(G)]\cong \bC[t_1^{\pm 1}, \cdots, t_{m+2g}^{\pm 1}].$ Thus $\Char(G)\cong \spec \Gamma_G \cong (\bC^*)^{m+2g}$. 

Let $\sF\in D^b_c(G,\bC)$ be a  bounded constructible complex of $\bC$-sheaves on $G$ with cohomology jump loci $\sV^i(G,\sF)$. By using the linear structure of irreducible components of the cohomology jump loci (see Theorem \ref{linear}), we introduce refined notions of {\it (semi)abelian codimensions} 
$\codim_{sa} \sV^i(G,\sF)$ and $\codim_{a} \sV^i(G,\sF)$, see Definition \ref{GA}.

The main result of our paper asserts that the position of a bounded $\bC$-constructible complex on $G$ with respect to the perverse $t$-structure on $D^b_c(G,\bC)$ can be detected by the (semi)abelian codimension of its cohomology jump loci. This result provides a complete description of $\bC$-perverse sheaves on a semi-abelian variety $G$ in terms of their cohomology jump loci, and generalizes and unifies Schnell's corresponding result \cite[Theorem 7.4]{Sch} for perverse sheaves on abelian varieties, as well as Gabber-Loeser's description \cite{GL} of perverse sheaves on complex affine tori. Specifically, we prove the following (see Theorem \ref{characterization}).
\begin{theorem} \label{icharacterization}
Let $\sF\in D^b_c(G, \bC)$ be a bounded $\bC$-constructible complex on $G$. Then we have
\begin{enumerate}
\item[(a)] $\sF\in \,^p D^{\leq 0}(G, \bC) \ \iff \ \codim_{a} \sV^i(G, \sF) \geq i \text{ for any } i\geq 0,$
\item[(b)] $\sF\in \,^p D^{\geq 0}(G, \bC) \ \iff \ \codim_{sa} \sV^i(G, \sF) \geq -i \text{ for any } i\leq 0.$
\end{enumerate}
Thus, $\sF$ is a $\bC$-perverse sheaf on $G$ if and only if the following two conditions are satisfied:
\begin{enumerate}
\item $\codim_{a} \sV^i(G, \sF) \geq i$ for any $i\geq 0$,
\item $\codim_{sa} \sV^i(G, \sF) \geq -i$ for any $i\leq 0$. 
\end{enumerate}
\end{theorem}

As pointed out in Remark \ref{coinc}, cohomology jump loci of a constructible complex $\sF$ on $G$ can be reduced to investigating the corresponding cohomology jump loci of $\sM_*(\sF)$, with 
$\sM_*: D^b_c(G, \bC)\to D^b_{coh}(\Gamma_G)$ the Mellin transformation (see Definition \ref{Mellin}).
For complex affine tori, the image $\sM_*(\sP)$ of a perverse sheaf $\sP$ is a single coherent sheaf (cf.  \cite{GL}), while for an abelian varieties it is a perverse coherent sheaf (cf. \cite{Sch}).  It is therefore natural to ask whether there exists a $t$-structure ${^{coh}}\tau$ on $D^b_{coh}(\Gamma_G)$, which ``glues'' the standard one on the torus part with the perverse coherent one on the abelian variety, so that $\sM_*(\sP)$ sits inside the heart of this $t$-structure (or, equivalently, $\sM_*:(D^b_c(G, \bC), {^p}\tau)\to (D^b_{coh}(\Gamma_G),{^{coh}}\tau)$ is $t$-exact, with ${^p}\tau$ denoting the perverse $t$-structure on constructible complexes). 
As discussed in Remark \ref{tstructure}, the answer is negative in general, if one restricts to the family of t-structures constructed in \cite{AB}.

In addition  to Theorem \ref{icharacterization}, properties of the Mellin transformation (Theorem \ref{semi}) are used here to generalize our results from \cite{LMW3}, and show that cohomology jump loci of perverse sheaves on semi-abelian varieties {\it propagate}:
\bt\label{propa}  The cohomology jump loci of any $\bC$-perverse sheaf $\sP$ on $G$ satisfy the following propagation property:
$$
 \sV^{-m-g}(G,\sP) \subseteq \cdots \subseteq \sV^{-1}(G,\sP) \subseteq \sV^{0}(G,\sP) \supseteq \sV^{1}(G,\sP) \supseteq \cdots \supseteq \sV^{g}(G,\sP).$$
Furthermore, $\sV^i(G,\sP)= \emptyset$ if $i\notin [-m-g, g]$.
\et

As an immediate consequence of Theorems \ref{icharacterization} and \ref{propa}, we get the following {known result}.
\bc\label{gv-intro} Let $\sP$ be a $\bC$-perverse sheaf on a semi-abelian variety $G$. Then $\sP$ satisfies the following properties:
\begin{enumerate}
\item[(i)] {\it Generic vanishing}: there exists a non-empty Zariski open subset $U \subset \Char(G) $ such that,  for any closed point $\rho\in U$,  $H^{i}(G, \sP\otimes_{\bC} L_\rho)=0$ for all $i\neq 0$.
\item[(ii)] {\it Signed Euler characteristic property}:  $$  \chi(G,\sP)\geq 0.$$
Moreover, the equality holds if and only if $\sV^0(G,\sP) \neq \Char(G)$.
\end{enumerate}
\ec
\br In the semi-abelian context,  the codimension lower bound (Theorem \ref{icharacterization}) and the propagation (Theorem \ref{propa}) for the cohomology jump loci of perverse sheaves are, to our knowledge, new and unexplored properties.
They explain in a more conceptual way the generic vanishing and signed Euler characteristic properties of Corollary \ref{gv-intro}. A propagation property for simple perverse sheaves on abelian varieties was previously obtained in \cite{We16b}. 
In the coherent setting, a similar propagation property was proved, e.g.,  in \cite[Proposition 3.14]{PP11}, for cohomology support loci of GV-sheaves on abelian varieties. 

The generic vanishing property for perverse sheaves on semi-abelian varieties was previously obtained by different methods in  \cite[Theorem 2.1]{Kra} and \cite[Theorem 4.3]{LMW2} (for the abelian context see \cite[Theorem 1.1]{KW}, \cite[Corollary 7.5]{Sch},  \cite[Vanishing Theorem]{We16}, \cite[Theorem 1.1]{BSS}). The signed Euler characteristic property  (which is an immediate consequence of generic vanishing) is originally due to Franecki and Kapranov \cite[Corollary 1.4]{FK}.  
\er

The following description of simple perverse sheaves with zero Euler number plays an essential role in the proof of Theorem \ref{icharacterization}. It also provides a unification and generalization to the semi-abelian context of similar statements in the abelian case (cf. \cite[Theorem 2]{We0}, \cite[Proposition 10.1(a)]{KW}, \cite[Theorem 7.6]{Sch}) and in the affine torus case (cf. \cite[Theorem 5.1.1]{GL}), respectively. 
\bt \label{isimple}    If $\sP$ is a simple perverse sheaf on $G$ with $\chi(G, \sP)=0$, then there exists a positive dimensional semi-abelian subvariety $G^{\prime \prime}$ of $G$, a rank-one $\bC$-local system $L_\rho$ on $G$ and a simple perverse sheaf $ \sP^\prime$
on $G^\prime=G/G^{\prime \prime}$ with $\chi(G', \sP')\neq 0$, such that  
$$ \sP\cong L_\rho \otimes_{\bC} f^* \sP^\prime[\dim G^{\prime \prime}],$$ with $f: G\to G'=G/G^{\prime \prime}$ denoting the quotient map.
Moreover,  $\sV^0(G,\sP)$ is an irreducible linear subvariety. 
\et

\subsection{Applications} Additional motivation for investigating perverse sheaves on semi-abelian varieties is provided by their wide range of applications,
including to the study of cohomology jump loci of smooth quasi-projective varieties, for understanding the topology of the Albanese map, as well as in the context of homological duality of complex algebraic varieties. We sample here several such applications, for more details see Section \ref{appl}. The interested reader may also consult \cite{LMW19}, which gives a brief survey of results and applications contained in the present paper and in \cite{LMW3}.

In relation to cohomology jump loci of smooth varieties, we have the following (see Corollary \ref{manifold} for a more general statement).
\bc\label{icor}
Let  $X$ be a  smooth complex quasi-projective variety of dimension $n$ with Albanese map $\alb:X \to \Alb(X)$. Suppose $R\alb_* \bC_X[n]$ is a perverse sheaf on $\Alb(X)$ (e.g., $\alb$ is proper and semi-small). Then the cohomology jump loci $\sV^i(X)$ have the following properties:
\begin{itemize}
\item[(1)] {\it Propagation property}:
$$ 
\{ \mathbf{1} \} = \sV^{0}(X)  \subseteq \cdots \subseteq \sV^{n-1}(X) \subseteq   \sV^{n}(X) \supseteq \sV^{n+1}(X) \supseteq \cdots \supseteq \sV^{2n}(X) .
$$
\item[(2)] {\it Codimension lower bound}: for any $ i\geq 0$, $$ \codim_{sa} \sV^{n-i}(X)\geq i \ \ {\rm and} \  \codim_{a} \sV^{n+i}(X) \geq i.$$
\item[(3)]For generic $\rho \in \Char(X)$, $H^{i}(X,  L_\rho)=0$ for all $i \neq n$.
\item[(4)] 
$b_i(X)>0$ for any $i\in [0, n]$,
and 
$b_1(X)\geq n$.
\end{itemize} 
\ec
\br A class of smooth complex quasi-projective varieties with proper and semi-small Albanese map is given in Example \ref{ample}.  Note also that by Theorem \ref{icharacterization}, the codimension lower bound of Corollary \ref{icor}(2) is equivalent to the fact that $R\alb_* \bC_X[n]$ is a perverse sheaf on $\Alb(X)$.
\er

As another application, to the topology of the Albanese map, we get the following generalization of \cite[Theorem 2.1]{Wa17}, see Corollary \ref{isol}. 

\bc  Let $X$ be an $n$-dimensional smooth complex quasi-projective variety with Albanese map $\alb : X \to \Alb(X)$.  If $\bigcup_{i=0}^{2n} \sV^i(X)$ contains an isolated point, then $\alb$ is dominant. 
\ec

The following generalization of \cite[Corollary 2.6]{BC} gives a topological characterization of semi-abelian varieties (see Proposition \ref{topsa}):
\bc Let $X$ be a smooth quasi-projective variety with proper Albanese map (e.g., $X$ is projective), and assume that $X$ is homotopy equivalent to a torus. Then $X$ is isomorphic to a semi-abelian variety. 
\ec

Finally, let us indicate here an application of our results to homological duality. The concept of {\it abelian duality space} was introduced by Denham-Suciu-Yuzvinsky in \cite{DSY} as an abelian version of the Bieri-Eckmann duality spces \cite{BE}, see Subsection \ref{ads} for a definition and properties.  Examples of abelian duality spaces were constructed in \cite[Theorem 4.11]{LMW3} via algebraic maps to complex affine tori. 
Here we provide generalizations of this construction to the semi-abelian setting. For example, we prove the following result (see Theorem \ref{abelianduality}).
\bc  Let $X$ be an $n$-dimensional smooth complex quasi-projective variety, which is homotopy equivalent to an $n$-dimensional CW complex (e.g., $X$ is affine).  Suppose the Albanese map $\alb$ is proper and semi-small, or $\alb$ is quasi-finite. Then $X$ is an abelian duality space of dimension $n$. 
\ec
As a consequence, very affine manifolds and complements of essential hypersurface arrangements in projective manifolds are abelian duality spaces (see  Examples \ref{va} and \ref{ample}).

\subsection{Summary}
The paper is organized as follows.
  
In Section \ref{alg}, we recall the definition of several algebraic notions (including Fitting ideals, cohomology jumping ideals and cohomology jump loci of a bounded complex of finitely generated modules) and prove several preparatory commutative algebra results.
  
In Section \ref{pta}, we recall the relevant results for perverse sheaves on complex affine tori and abelian varieties, which provide a motivation for our work. 

In Section \ref{top}, we study properties of the Mellin transformation functor for perverse sheaves on a complex semi-abelian variety. In particular, we prove the propagation property of Theorem \ref{propa}. 

Section \ref{zeroeu} is devoted to characterizing simple $\bC$-perverse sheaves with vanishing Euler number  on semi-abelian varieties. In particular, we prove Theorem \ref{isimple}.

In Section \ref{cpv}, we prove Theorem \ref{icharacterization} on  the characterization of $\bC$-perverse sheaves on semi-abelian varieties in terms of their cohomology jump loci.

Finally, Section \ref{appl} is devoted to applications. 



\medskip

\textbf{Acknowledgments.} We thank J\"org~Sch\"urmann, Dima~Arinkin and Nero~Budur for valuable discussions. We are also grateful to Rainer Weissauer and Thomas Kr\"amer for bringing related works and questions to our attention. The first author  thanks the Mathematics Department at the University of Wisconsin-Madison  for hospitality during the preparation of this work. The first author is partially supported by the starting grant KY2340000123 from University of Science and Technology of China, the project ``Analysis and Geometry on Bundles" of Ministry of Science and Technology of the People’s Republic of China and Nero Budur's research project G0B2115N from the Research Foundation of Flanders. The second author is partially supported by the Simons Foundation Collaboration Grant \#567077 and by the Romanian Ministry of National Education, CNCS-UEFISCDI, grant PN-III-P4-ID-PCE-2016-0030. The third author is partially supported by the NSF grant DMS-1701305 and a Sloan Fellowship.


\section{Cohomology jump loci of a complex of $R$-modules}\label{alg}
Let $R$ be a Noetherian domain, and denote by $\spec R$ the maximal spectrum of $R$. Let $E^\ubul$ be a bounded above complex of $R$-modules with finitely generated cohomology. In this section, we recall the notion of cohomology jump loci for the complex $E^\ubul$, and discuss some preparatory results in commutative algebra. 

By a construction of Mumford (see \cite[III.12.3]{Ha}), there exists a bounded above complex $F^\ubul$ of finitely generated {\it free} $R$-modules, which is quasi-isomorphic to $E^\ubul$. 

\bd For any integer $k$ and a map $\phi$ of finitely generated free $R$-modules, let $I_{k}\phi$ denote the {\it $k$-th determinantal ideal} of $\phi$ (i.e., the ideal of minors of size $k$ of the matrix of $\phi$), see \cite[p.492-493]{E}.  
For a bounded above complex $E^\ubul$ of $R$-modules, with finitely generated cohomology, and for $F^\ubul$ a bounded above finitely generated free resolution of $E^\ubul$ as above, the {\it degree $i$ Fitting ideal} of $E^\ubul$ is defined as:
$$I^i(E^\ubul)= I_{\rank \partial^{i}}( \partial^i),$$
and the {\it  degree $i$ jumping ideal} of $E^\ubul$ is defined as: $$J^{i} (E^\ubul)=I_{\rank(F^{i})}(\partial^{i-1}\oplus \partial^{i}),
$$
where  $\partial^{i-1}: F^{i-1}\to F^{i}$ and $\partial^{i}: F^{i}\to F^{i+1}$ are differentials of the complex $F^\ubul$. 

We define the (reduced) {\it $i$-th cohomology jump locus} of $E^\ubul$ as the algebraic subset of $\spec R$ associated to $J^i(E^\ubul)$, that is,
$$  \sV^{i} (E^\ubul) :=  V\left(\sqrt{J^i(E^\ubul)}\right)\subset \spec R,
$$ 
where $\sqrt{I}$ denotes the radical ideal of $I$. 
Notice that $\sV^{i} (E^\ubul)$ is naturally isomorphic to the reduced induced closed subscheme of $\spec R/J^i(E^\ubul)$. 
\ed 
It is known that $I^i(E^\ubul)$,  $J^i(E^\ubul)$ and $\sV^i(E^\ubul)$ do not depend on the choice of the finitely generated free resolution $F^\ubul$ of $E^\ubul$; see \cite[Section 20.2]{E} and \cite[Section 2]{BW}.

\br\label{alt}
The closed points of $\sV^{i} (E^\ubul)$ can also be described as follows:
$$ \sV^{i} (E^\ubul)= \{ \rho \in \spec R \mid H^i (F^\ubul \otimes _R R/\rho) \neq 0 \},$$ 
 with $F^\ubul$ a bounded above finitely generated free resolution of $E^\ubul$,  see \cite[Corollary 2.5]{BW}.\er
 
For the rest of this section, we focus on a special class of complexes of finitely generated free $R$-modules of finite length, of the form:
$$ 0\to F^{-k} \to \cdots \to F^{i-1} \overset{\partial^{i-1}}{\to} F^{i} \overset{\partial^{i}}{\to} F^{i+1} \to  \cdots \to F^{\ell-1} \overset{\partial^{\ell-1}}{\to}  F^\ell \to 0,$$
with $k, \ell\geq 0$. In what follows, we work under the following assumption: 
\begin{assumption}\label{2vanishing} $H^i(F^\ubul)=0$ and $H^i(\homo_R (F^\ubul, R))=0$,
for any $i<0$.
\end{assumption}  \noindent Here, the complex $\homo_R (F^\ubul, R)$ is given by: 
\begin{center}
$0\to \homo_R (F^{\ell}, R) \to \cdots \to \homo_R (F^{i+1}, R) \overset{\partial^{i,\vee}}{\to} \homo_R (F^{i}, R) \overset{\partial^{i-1,\vee}}{\to} \homo_R (F^{i-1}, R) \to  \cdots \to  \homo_R (F^{-k}, R) \to 0,$
\end{center} 
with $\homo_R(F^i, R)$ placed in degree $-i$.

\medskip

This section is devoted to proving the following result:
\bp \label{prop}   Let $F^\ubul$ be a finite length complex of finitely generated free $R$-modules, satisfying Assumption \ref{2vanishing}. Then, for any $i\neq 0$, we have that:  $$\sqrt{ I^{i}(F^\ubul)}= \sqrt{ J^{i}(F^\ubul)},$$
where $\sqrt{I}$ denotes the radical ideal of $I$. Moreover, the following properties hold:
\begin{enumerate}
\item[(i)]    {\it Propagation property}:
$$\sV^{-k}(F^\ubul) \subseteq \cdots \subseteq \sV^{-1}(F^\ubul) \subseteq \sV^{0}(F^\ubul) \supseteq \sV^{1}(F^\ubul) \supseteq \cdots \supseteq \sV^{\ell}(F^\ubul).
$$
\item[(ii)] {\it Depth lower bound}: for any $i\geq 0$, $$ \depth J^{\pm i}(F^\ubul) \geq i.$$
\item[(iii)]  Assume that  $V$ is an irreducible component of $\sV^0(F^\ubul)$ and let $\fp$ denote the corresponding prime ideal. Let $\fp_\fp$ denote the maximal ideal in the local ring $R_\fp$. 
If  $d= \depth \fp_\fp$, then there exists an interval $[a_1,a_2]$ with $0\in [a_1,a_2]$ and $a_2-a_1\geq d$ such that $V$ is an irreducible component of $\sV^{i}(F^\ubul)$ exactly for $i\in [a_1, a_2]$. 
\end{enumerate}
\ep

Before proving Proposition \ref{prop}, we recall the following two results from \cite[Theorem 20.9, Corollary 20.12, Corollary 20.14]{E} and \cite[Lemma 2.3, Proposition 2.5]{LMW3}.
\bl \label{E}    A complex of finitely generated free $R$-modules 
$$ F^\ubul: \  0\to F^{-k} \to \cdots \to F^{i-1} \overset{\partial^{i-1}}{\to} F^{i} \overset{\partial^{i}}{\to} F^{i+1} \to  \cdots \overset{\partial^{-1}}{\to} F^{0} $$ is exact if and only if  
the following two properties hold for any $ i\leq -1$:
\begin{itemize}
\item[(i)]   $$\rank F^{i}= \rank \partial^{i} +\rank \partial^{i-1}  .$$ 
\item[(ii)]  
 $$\depth I^{i}(F^\ubul) \geq - i \  .$$
\end{itemize}
\el
\bl \label{LMW} 
Let $F^\ubul$ be an exact complex as in Lemma \ref{E}. Then for any $i\leq -1$,
$$\sqrt{ I^{i}(F^\ubul)} \subset \sqrt{ I^{i-1}(F^\ubul)} \text{ and } \sqrt{I^{i}(F^\ubul)}= \sqrt{ J^{i}(F^\ubul)}.$$
\el

We will also make use of the following proposition.
\begin{prop} \label{W}
Let $R$ be a noetherian local domain. Let 
$$0\to F^{-k}\to\cdots\to F^{-2}\xrightarrow[]{\partial^{-2}} F^{-1}\xrightarrow[]{\partial^{-1}}F^0\xrightarrow[]{\partial^0}F^1$$
be a complex of finitely generated free $R$-modules, with $H^i(F^\ubul)=0$ for $i<0$. Let $r_i$ be the rank of $F^i$.  If $I_{r_0}(\partial^{-1}\oplus \partial^{0})=R$, 
then $I_{r_{i}}(\partial^{i}\oplus \partial^{i-1})=R$ for any $i<0$. \end{prop}

\begin{proof}
First, we recall our convention that $I_0(\phi)=R$ for any homomorphism $\phi$ between free $R$-modules and if $j$ is larger than the size of $\phi$, then $I_j(\phi)=0$.
By definition, 
$$J^0(F^\ubul)=I_{r_0}(\partial^{-1}\oplus \partial^{0})=\sum_{0\leq j\leq r_0} I_{j}(\partial^{-1})\cdot I_{r_0-j}(\partial^0).$$
Since $R$ is a local ring, $\sum_{0\leq j\leq r_0} I_{j}(\partial^{-1})\cdot I_{r_0-j}(\partial^0)=R$ implies that $I_{j}(\partial^{-1})\cdot I_{r_0-j}(\partial^0)=R$ for some $j$, or equivalently $I_{j}(\partial^{-1})=I_{r_0-j}(\partial^0)=R$. Fix this particular $j$. Then $j\leq \rank(\partial^{-1})$ and $r_0-j\leq \rank(\partial^0)$. On the other hand, since $(F^\ubul, \partial^\ubul)$ is a complex, we have
$$\rank(\partial^{-1})+\rank(\partial^0)\leq r_0.$$
Therefore, $\rank(\partial^{-1})=j$ and  $\rank(\partial^0)=r_0-j$. 

Since $I_{j}(\partial^{-1})=R$, $\partial^{-1}$ has a $j\times j$ minor, whose determinant is invertible. By a  suitable base change, this minor is equal to the size $j$ identity matrix. Without loss of generality, we may assume that this minor is located in the upper left corner. Then, by a further base change, we can assume that away from the top-left $j\times j$ corner, the top $j$ rows and the left $j$ columns of $\partial^{-1}$ are all zero. Now, since $\rank(\partial^{-1})=j$, the lower right corner of $\partial^{-1}$ is also zero. Thus, the image of $\partial^{-1}$ is a free direct summand of $F^0$. Therefore, a non-canonical splitting $F^0\cong \im(\partial^{-1})\oplus F^0/\im(\partial^{-1})$ induces a splitting of $F^\ubul$ into two complexes, one with only nonnegative degree terms 
\[ 0\to F^0/\im(\partial^{-1}) \overset{\partial_0}{\lra} F^1 \, \]
and one with only nonpositive degree terms
\[ 0\to F^{-k}\to\cdots\to F^{-2}\xrightarrow[]{\partial^{-2}} F^{-1}\xrightarrow[]{\partial^{-1}} \im(\partial^{-1})  \lra 0.\]
The complex with nonpositive degree terms is quasi-isomorphic to zero by assumption. Now, it follows that $I_{r_i}(\partial^{i}\oplus \partial^{i-1})=R$ for all $i<0$.
\end{proof}

We have now all the ingredients for proving Proposition \ref{prop}.
\begin{proof}[Proof of Proposition \ref{prop}]
Consider the truncated complex $F^{\leq 0}$: 
$$ 0\to F^{-k} \to \cdots \to F^{i-1} \overset{\partial^{i-1}}{\to} F^{i} \overset{\partial^{i}}{\to} F^{i+1} \to  \cdots \to F^{-1} \overset{\partial^{-1}}{\to} F^0  .$$
Then for any  $i\leq -1$,
$$ J^{i}(F^\ubul)=J^{i}(F^{\leq 0}).$$
Note that the depth of an ideal only depends on its radical ideal (\cite[Corollary 17.8]{E}).
Therefore, for any $i\leq -1$, we have that $$\depth  J^{i}(F^\ubul) =\depth \sqrt{ J^{i} (F^\ubul)}=\depth\sqrt{ I^{i}(F^\ubul)}=\depth I^{i}(F^\ubul) \geq \vert i\vert. $$
So if $i<0$, the claims in $(i)$ and $(ii)$ follow from Lemma \ref{E} and Lemma \ref{LMW}.  

The similar claims  for $i>0$ follow from a dual argument, based on the assumption  that $H^i(\homo_R (F^\ubul, R))=0$ for $i<0$.  Note that $\homo_R (\homo_R(F^\ubul, R)) \cong F^\ubul$. This completes the proof of $(ii)$.

To finish the proof of $(i)$, it remains to show that $\sV^{-1}(F^\ubul) \subseteq \sV^{0}(F^\ubul) \supseteq \sV^{1}(F^\ubul)$. The claim is obviously true if $\sV^0(F^\ubul)= \spec R$. If $\sV^0(F^\ubul)\neq \spec R$, we prove the claim by localization.
Let $\fm$ be a maximal ideal such that $\fm \notin \sV^0(F^\ubul)$. Let $R_\fm$ be the localization of $R$ at $\fm$, and denote $F^\ubul\otimes_R R_\fm$ by $F^\ubul_\fm$. Since  localization is an exact functor, we have that $H^i(F^\ubul_\fm)=0$ and $H^i(\homo_{R_\fm}(F^\ubul_\fm, R_\fm))=0$ for any $i<0$.
Note that $\fm \notin \sV^0(F^\ubul)$ if and only if  
 $\sV^0(F^\ubul_\fm)= \emptyset$, i.e., $J^0(F^\ubul_\fm)=R_\fm$.  On the other hand, it follows from Proposition \ref{W} that  $J^{-1}(F^\ubul_\fm)=R_\fm$, i.e., $\fm\notin \sV^{-1}(F^\ubul)$. Therefore, $\sV^{-1}(F^\ubul) \subseteq \sV^{0}(F^\ubul) $. The inclusion $\sV^{0}(F^\ubul) \supseteq \sV^{1}(F^\ubul)$ follows from a dual argument.

\medskip

Let us now prove $(iii)$. Suppose that $V\neq \spec R$, otherwise the claim is automatic.  Let $\fp$ denote the prime ideal associated to the irreducible component $V$. Consider the corresponding local ring $R_\fp$. Let $F^\ubul_\fp$ be the localization of the complex $F^\ubul$ at $\fp$. Since  localization is an exact functor, by Assumptions \ref{2vanishing}, we have $H^i(F^\ubul_\fp)=0$ and $H^i(\homo_{R_\fp}(F^\ubul_\fp, R_\fp))=0$ for any $i<0$. By
Lemma \ref{E}, for any $i\neq 0$, \be \label{local} \rank F_\fp^{i}= \rank \partial_\fp^{i} +\rank \partial_\fp^{i-1} .
\ee

Note that  $V$ is an irreducible component of $ \sV^0(F^\ubul)$ if and only if  $\sV^0(F^\ubul_\fp)= \{\fp_\fp\}$. Since $\sV^0(F^\ubul_\fp)\neq \spec R_\fp$, the equality (\ref{local}) also holds for $i=0.$ The propagation property $(i)$ implies that there exists an interval  $[a_1,a_2]$ 
containing $0$, such that $ \sV^i(F^\ubul_\fp)=\{\fp_\fp\}$ for any
 $i\in [a_1,a_2]$, 
and $\sV^i(F^\ubul_\fp)=\emptyset$ for any $i\notin[a_1,a_2]$. 
Note that the depth of $J^0(F^\ubul_\fp)$ is $d$. Thus, by equality (\ref{local}) and Lemma \ref{E}, we have $H^i(F^\ubul_\fp)=0$  for  $i<a_1+d$.

Since $\sV^i(F^\ubul_\fp)=\emptyset$ for  $i>a_2$, we have that $J^i(F^\ubul_\fp)=R_\fp$.  By using the same  argument as in the proof of Proposition \ref{W}, 
 it follows that $F^\ubul_\fp$  is exact for $i>a_2$. If $a_2<a_1+d$, then the whole complex $F^\ubul_\fp$ is exact, hence quasi-isomorphic to the zero complex. But this contradicts the assumption $\sV^0(F^\ubul_\fp)=\{\fp_\fp\}$.  So $a_2\geq a_1+d$.
\end{proof}
\br \label{half}  Let $F^\ubul$ be a finite length complex of finitely generated free $R$-modules such that $H^i(F^\ubul)=0$ for any $i<0$. The above proof of Proposition \ref{prop}$(i)$ yields that
 $$\cdots \subseteq \sV^{-2}(F^\ubul)   \subseteq \sV^{-1}(F^\ubul) \subseteq \sV^{0}(F^\ubul).$$ 
\er

\br \label{projective_free} Let $\kk$ 
be a principal ideal domain (PID).  Assume that $R=\kk[t_1^{\pm 1}, \cdots, t_N^{\pm 1}]$ for some positive integer $N$. 
 Note that any finitely generated projective $R$-module  is free \cite[Theorem 8.13]{BG}, hence every complex of $R$-modules with bounded  finitely generated cohomology  admits a free resolution of finite length. Moreover,  since $R$ is a Cohen-Macaulay ring in this case, the depth of  any ideal in $R$ coincides with its codimension.
\er


\section{Perverse sheaves on complex affine tori and abelian varieties}\label{pta}
In this section, we introduce the Mellin transformation functor and recall some properties of the cohomology jump loci of perverse sheaves on complex affine tori and abelian varieties, which will later on be generalized to the context of semi-abelian varieties. 

For any complex algebraic variety $X$ and any commutative Noetherian ring $R$, we denote by  $D^{b}_{c}(X,R)$ the derived category of bounded cohomologically constructible $R$-complexes of sheaves on $X$. 

Throughout this section, we fix a coefficient field $\kk$, e.g.,  $\bC$ or $\Fp=\bZ/p\bZ$ (for a prime $p$). Let  $\Perv(X,\kk)$ denote the category of perverse sheaves with $\kk$-coefficients on $X$.

\subsection{Mellin transformation and cohomologi jump loci}\label{MT}

Let $G$ be a complex semi-abelian variety, i.e., a complex algebraic group $G$ which is an extension
$$1 \to T \to G \to A \to 1,$$
where $A$ is an abelian variety of complex dimension $g$ and $T\cong(\bC^*)^m$ is an affine algebraic torus of complex dimension $m$.
 Set $$\Gamma_G:=\kk[\pi_1(G)]\cong \kk[t_1^{\pm 1}, \cdots, t_{m+2g}^{\pm 1}].$$ 
Let $\sL_G$ be the rank-one local system of $\Gamma_G$-modules on $G$ associated to the tautological character $\tau:\pi_1(G) \to \Gamma_G^*$, which maps the generators of $\pi_1(G)$ to the multiplication by the corresponding variables of the Laurent polynomial ring $\Gamma_G$.

 \bd \label{Mellin}\cite{GL} The {\it Mellin transformation functors (also called  Fourier transforms in \cite[Definition 2.5]{BSS})} $\sM_{\ast}, \sM_!: D^{b}_{c}(G, \kk) \to D^{b}_{coh}(\Gamma_G)$ are defined by
$$\sM_\ast(\sF) := Ra_\ast( \sL_G\otimes_{\kk}\sF) \ , \ \ \ \sM_!(\sF) := Ra_!( \sL_G\otimes_{\kk}\sF), $$
where  $D^{b}_{coh}(\Gamma_G)$ denotes the derived category of bounded coherent complexes of $\Gamma_G$-modules, and $a :G\to pt$ is  the constant map to a point. \ed

\bd For any $\kk$-constructible complex $\sF\in D^b_c(G,\kk)$, 
the {\it cohomology jump loci of $\sF$} are defined as:
\be\label{cjl}\sV^i(G,\sF): = \{\rho \in \spec \Gamma_G \mid H^i(G, \sF \otimes_\kk L_\rho) \neq 0\},\ee
where $L_{\rho}$ is the rank-one local system of $\kk_{\rho}$-vector spaces on $G$ associated to the maximal ideal $\rho$ of $\Gamma_G$, with $\kk_{\rho}=\Gamma_G/\rho$ the residue field of $\rho$.\ed

\br\label{coinc}
The relation between cohomology jump loci and the Mellin transformation is provided by the identification
\be\label{impeq} \sV^i(G,\sF)=\sV^i(\sM_\ast(\sF)),\ee
with the right-hand side defined as in Section \ref{alg}.
This is a consequence of the following isomorphism obtained by using the projection formula and Remark \ref{alt} (e.g., see the proof of \cite[Theorem 3.3]{LMW3}, and compare also with \cite[Lemma 2.6]{BSS} and \cite[Proposition 13.4]{Sch}):
\be
H^i(G, \sF \otimes_\kk L_\rho)\cong H^i(\sM_\ast(\sF) \otimes_{\Gamma_G} \kk_\rho).
\ee
\er

\subsection{Perverse sheaves on complex affine tori}

If $G$ is a complex affine torus, i.e., $G=T$, the following result was proved by Gabber-Loeser \cite[Theorem 3.4.1 and Theorem 3.4.7]{GL} in the $\ell$-adic context, and then extended to the present form in   \cite[Theorem 3.2]{LMW3}.
\bt \label{GL} Let $\kk$ be a fixed field. The Mellin transformation functor $\sM_\ast$ is $t$-exact, i.e., for any $\kk$-perverse sheaf $\sP$ on $T$, we have that $H^{i}(\sM_{\ast}(\sP))=0$ for $i\neq 0$. Moreover, a constructible complex $\sF\in D^b_c(T, \kk)$ is perverse if and only if $\sM_\ast(\sF)$ is isomorphic to a complex concentrated in degree zero. 
\et

\subsection{Perverse sheaves on abelian varieties}
If $G$ is a complex abelian variety, i.e., $G=A$, one has the following results proved by Bhatt-Schnell-Scholze and Schnell. 

\bt\label{Schnell} \cite[Theorem 7.4]{Sch}
A $\bC$-constructible complex $\sF\in D^b_c(A,\bC)$ is perverse on the abelian variety $A$ if and only if for any $i \in \bZ$, 
$$ \codim \sV^{i}(A,\sP) \geq \vert 2i \vert.$$
\et

\bt \label{BSS}  \cite[Proposition 2.7]{BSS} Let $\kk$ be a fixed field.  For any $\sP\in \Perv(A,\kk)$, we have that $\sM_*(\sP) \in 
D^{\geq 0} (\Gamma_A)$ and $D_{\Gamma_A}(\sM_*(\sP))  \in D^{\geq 0} (\Gamma_A) $, i.e., 
\begin{center}
$H^{i}(\sM_{\ast}(\sP))=0$  and  $H^i (D_{\Gamma_A} (\sM_*(\sP))) =0$ for all $i<0$.
\end{center}
Here $D_{\Gamma_A}(-):=\rhomo_{\Gamma_A} (-,\Gamma_A) $ is the dualizing functor for the ring $\Gamma_A$.
\et 

The codimension lower bound for cohomology jump loci of perverse sheaves, given in Theorem \ref{Schnell}, was more recently reproved in \cite[Theorem 3.1]{BSS} by using the Hard Lefschetz theorem. 
\bt \label{HL} ({\bf Hard Lefschetz}, \cite{Moc}) Let $\kk$ be a fixed field of characteristic zero.  If $c\in  H^2(A; \kk)$ is the Chern class of an ample line bundle (ignoring twists) and if $\sP\in \Perv(A, \kk)$ is semi-simple, then the cup product map 
$$ 
H^{-i}(A, \sP) \overset{c^i}{\to}  H^i(A, \sP) $$
is an isomorphism for any $i>0$.
\et
As a direct application of the Hard Lefschetz theorem, one also has the following.
\bc Let $\kk$ be a fixed field of characteristic zero.  For any semi-simple perverse sheaf $\sP\in \Perv(A,\kk)$ and  for any integer $i$, we have that
  $$\sV^{-i} (A, \sP) = \sV^{i} (A, \sP).$$ 
\ec

A more refined result on the nonvanishing of cohomology groups for simple perverse sheaves on abelian varieties was obtained by Weissauer in \cite{We16b}.

\section{Perverse sheaves on semi-abelian varieties}\label{top}
In this section, we employ the Mellin transformation to study properties of the cohomology jump loci of  $\kk$-perverse sheaves on a semi-abelian variety, with $\kk$ a fixed field of coefficients.

In order to prove the main results of this section, we need a few preparatory facts,  analogous to \cite{GL}. As in \cite{FK}, by choosing a splitting $T\cong (\bC^*)^m$, we can write the semi-abelian variety $G$ as
$$G=G_1\times_A G_2\times_A \cdots \times_A G_m$$
where each $G_k$ is an extension of the abelian variety $A$ by $\bC^*$. Then each $G_k$ is a principal $\bC^*$-bundle over $A$. Let  $$G'=G_2\times_A G_3\times_A \cdots \times_A G_m,$$ 
with $G'=A$ if $m=1$. 

\bl \label{iso} Let $f: G\to G'$ be the projection onto the semi-abelian variety $G'$, given by forgetting the first coordinate of $T$. Then for any $\sP \in D^b_c(G,\kk)$, we have:
$$\sM_* (Rf_* \sP) \cong \sM_*(\sP)  \stackrel{L}{\otimes}_{\Gamma_{G}} \Gamma_{G'},$$
and
$$\sM_! (Rf_! \sP) \cong \sM_!(\sP)  \stackrel{L}{\otimes}_{\Gamma_{G}} \Gamma_{G'},$$
where $\Gamma_{G'} \cong \Gamma_G/(t_1-1)$ as a $\Gamma_G$-module. 
\el
\begin{proof}
{The proof is very similar to that of \cite[Proposition 3.1.3(c)]{GL}. The latter is phrased in the $\ell$-adic setting on complex affine tori, but the arguments carry over to the topological context on semi-abelian varieties. We only indicate here the key points, while leaving  the details to the interested reader.

The main ingredient for the calculation of $\sM_! (Rf_! \sP)$ is the classical projection formula, by making use of  the isomorphism $f^* \sL_{G'} \cong \sL_G \stackrel{L}{\otimes}_{\Gamma_{G}} \Gamma_{G'}$ in $D^b_c(G, \Gamma_{G'})$. 
The computation of $\sM_* (Rf_* \sP)$ rests on a similar projection formula $$Rf_*(\sP \otimes_{\kk} f^*\sL_{G'}) \cong Rf_*(\sP) \otimes_{\kk} \sL_{G'},$$
which can be deduced by checking that the natural projection morphism (e.g., see \cite[Lemma 1.4.1]{Schu})
$$Rf_*(\sP) \otimes_{\kk} \sL_{G'} \to Rf_*(\sP \otimes_{\kk} f^*\sL_{G'}) $$ induces stalkwise isomorphisms (since $\sL_{G'}$ is locally constant). }
\end{proof}

We also need the following compatibility of the Mellin transformation with pullbacks (compare with \cite[Proposition 3.1.3(d)]{GL}):
\bl\label{pull}
Let $f: G\to G'$ be the projection onto the semi-abelian variety $G'$, as above, obtained by forgetting the first coordinate of $T$. Let $\sP' \in D^b_c(G',\kk)$. Then the following (non-canonical) isomorphisms hold in $D^b_{coh}(\Gamma_{G})$:
\begin{enumerate}
\item[(i)] $$\sM_! (f^* \sP') \cong \sM_!(\sP')[-2].$$
\item[(ii)] $$\sM_* (f^* \sP') \cong \sM_*(\sP')[-1].$$
\end{enumerate}
where $\Gamma_{G'}$ is viewed as a $\Gamma_G$-module via the isomorphism $\Gamma_{G'} \cong \Gamma_G/(t_1-1)$.
\el
\begin{proof} After fixing a splitting $G\cong G'\times \cc^*$ as topological spaces, we have the isomorphisms (compare with \cite[Proposition 3.1.1]{GL}): $$Rf_! \sL_{G}\cong \sL_{G'}[-2]$$ and  $$Rf_* \sL_{G}\cong \sL_{G'}[-1],$$
with $\sL_{G'}$ viewed as a $\Gamma_G$-sheaf. Here the shift on the right hand side corresponds to the $\cc^*$-factor in the splitting.
The desired isomorphisms follow then from the projection formula as in the proof of  \cite[Proposition 3.1.3(d)]{GL}. \end{proof}

Let $\sL_G^\vee $ denote the dual of the local system $\sL_G$.  Define $$ \sM_!^\vee( \sP):= Ra_! (\sP \otimes_{\kk} \sL_G^\vee).$$
As in \cite[Proposition 3.1.3(b)]{GL}, we have that 
\be  \label{dual}
D_{\Gamma_G} (\sM_*(\sP))\cong \sM_!^\vee(D \sP),
\ee 
where $D_{\Gamma_G}$ is the duality functor $\rhomo_{\Gamma_G}(-, \Gamma_G)$ in $D^{b}_{coh}(\Gamma_G)$.

 The following result is inspired by Gabber-Loeser's Theorem \ref{GL}  for perverse sheaves on a complex affine torus and Bhatt-Schnell-Scholze's Theorem \ref{BSS} 
 for perverse sheaves on abelian varieties.
\bt \label{semi}  Fix a field $\kk$. Let $\sP\in \Perv(G,\kk)$ be a perverse sheaf on the semi-abelian variety $G$. Then the following properties hold:
\begin{enumerate}
\item[(i)]  $\sM_*(\sP) \in 
D^{\geq 0} (\Gamma_G)$,  i.e., 
\begin{center}
$H^{i}(\sM_{\ast}(\sP))=0$ for all $i<0$. 
\end{center}
\item[(ii)]  $D_{\Gamma_G}(\sM_*(\sP))  \in D^{\geq 0} (\Gamma_G) $, i.e., 
\begin{center}
   $H^i (D_{\Gamma_G} (\sM_*(\sP))) =0$ for all $i<0$.
\end{center}
Here $D_{\Gamma_G} $ denotes as above the dualizing functor for the ring $\Gamma_G$.
\end{enumerate}
\et
 
\begin{proof}
The Mellin transformation commutes with field extensions, so $$\sM_*(\sP\otimes_\kk \overline{\kk})\cong \sM_*(\sP)\otimes_\kk \overline{\kk},$$
where $\overline{\kk}$ is the algebraic closure of $\kk$. 
Therefore, it suffices to prove the theorem in the case when $\kk$ is algebraically closed.

The abelian category of $\kk$-perverse sheaves is Artinian and Noetherian,  hence there is a well-defined notion of length of $\kk$-perverse sheaves.   By induction on the length of $\kk$-perverse sheaves, all claims in the statement can be reduced to the case of simple perverse sheaves.

Therefore, without any loss of generality, we can assume that  $\sP$ is a simple $\kk$-perverse sheaf on $G$, with $\kk$ an algebraically closed field. We prove the theorem by induction on the dimension of the torus $T$. 
If $\dim T=0$, then $G$ is an abelian variety, so both assertions follow directly from Theorem \ref{BSS}. For the induction step we proceed as follows.

\medskip
 
\noindent{\it Proof of Claim $(i)$.} \ 

Let $f: G\to G'$ be the projection onto the semi-abelian variety $G'$ as in Lemma \ref{iso}. Since the relative dimension of the affine morphism $f$ is $1$, the only possibly non-trivial perverse cohomology sheaves $^p \cH^i Rf_*(\sP)$ may appear in the range $i\in \{-1, 0\}$ (see, e.g., \cite[Corollary 5.2.14(ii), Theorem 5.2.16(i)]{D2}). Since $f$ is a smooth morphism of relative dimension one, it follows from \cite[Page 111]{BBD} that there is a canonical monomorphism of perverse sheaves
\be\label{mn1} f^*\left( ^p \cH^{-1} Rf_*(\sP)\right)[1]\hookrightarrow \sP.\ee
If $^p \cH^{-1} Rf_*(\sP)$ is non-zero then, since $\sP$ is simple, the monomorphism (\ref{mn1}) is an isomorphism:
$$f^*\left( ^p \cH^{-1} Rf_*(\sP)\right)[1]\cong \sP.$$
The desired claim follows in this case from Lemma \ref{pull}(ii), by using the induction hypothesis applied to the perverse sheaf $^p \cH^{-1} Rf_*(\sP)$ on $G'$.

On the other hand, if $^p \cH^{-1} Rf_*(\sP)$ is zero, then $Rf_*(\sP)$ is a perverse sheaf on $G'$.
Recall that by Lemma \ref{iso} we have  the isomorphism 
$$\sM_*(Rf_*(\sP))\cong \sM_*(\sP)\stackrel{L}{\otimes}_{\Gamma_{G}} \Gamma_{G'}.$$
Since $\Gamma_{G'}\cong \Gamma_G/(t_1-1)$, the complex $\Gamma_G\stackrel{t_1-1}{\longrightarrow}\Gamma_G$ is a free resolution of the $\Gamma_G$-module $\Gamma_{G'}$. Thus, 
\begin{equation}\label{tensor1}
\sM_*(Rf_*(\sP))\cong \sM_*(\sP){\otimes}_{\Gamma_G}(\Gamma_G\stackrel{t_1-1}{\longrightarrow}\Gamma_G).
\end{equation}
By the  induction hypothesis applied to the perverse sheaf $Rf_*(\sP)$ on $G'$, it follows that $H^i(\sM_*(Rf_*(\sP)))=0$ for $i<0$.
Hence by (\ref{tensor1}), we get that the multiplication by $t_1-1$
$$H^i(\sM_*(\sP))\stackrel{t_1-1}{\longrightarrow}H^i(\sM_*(\sP))$$
is surjective for $i<0$. 
Let $\fm\subset \Gamma_G$ be any maximal ideal  such that $(t_1-1)\in \fm$. Then,  by localization at $\fm$, we get that $$H^i(\sM_*(\sP))_\fm\stackrel{t_1-1}{\longrightarrow}H^i(\sM_*(\sP))_\fm$$
is surjective for $i<0$. 
Therefore, by Nakayama's Lemma for the local ring $(\Gamma_G)_\fm$, we get that 
$H^i(\sM_*(\sP))_\fm =0$ for $i<0$ if $(t_1-1)\in \fm$.
   
In general, fix a maximal ideal $\fm $ and assume that $(t_1-\lambda_1)\in \fm$ for some $\lambda_1 \in \kk^*$ (here we use the assumption that $\kk$ is algebraically closed). Consider the rank-one $\kk$-local system $L_{\lambda_1^{-1},1,\cdots,1}$ on $G$.  Then, as above,  $H^i(\sM_*(\sP\otimes L_{\lambda_1^{-1},1,\cdots,1}))_\fm =0$ for $i<0$ if $(t_1-1)\in \fm$, i.e.,  \begin{center}
$H^i(\sM_*(\sP))_\fm =0$ for $i<0$ if $(t_1-\lambda_1)\in \fm$.
\end{center}
Therefore, the vanishing $H^i(\sM_*(\sP))_\fm =0$ for $i<0$ holds for any maximal ideal $\fm\subset \Gamma_G$, which implies the desired result.

\bigskip
\noindent{\it Proof of Claim $(ii)$:} \ 

The proof of $(ii)$ is obtained by dualizing the arguments used for proving $(i)$. 
Indeed, we get by (\ref{dual})  that $D_{\Gamma_G} (\sM_*(\sP))\cong \sM_!^\vee(D \sP)$. Moreover, since $D \sP \in \Perv(G,\kk)$ if and only if $\sP \in \Perv(G,\kk)$, after replacing $\sP$ by its dual, it suffices to show that $H^i  (\sM_!^\vee(\sP)) =0$ for all $i<0$. 

If $\dim T=0$, then $G$ is an abelian variety, so the assertion follows from Theorem \ref{BSS}. 
For the induction step, 
let $f:G\to G'$ be the map considered above. The only possibly non-trivial perverse cohomology sheaves $^p \cH^i Rf_!(\sP)$ may appear in the range $i\in \{ 0, 1\}$ (see, e.g., \cite[Corollary 5.2.14(ii), Theorem 5.2.16(ii)]{D2}). Since $f$ is a smooth morphism of relative dimension $1$, one has a canonical epimorphism of perverse sheaves (cf. \cite{BBD})
\be\label{epi} \sP\twoheadrightarrow f^!\left( ^p \cH^{1} Rf_!(\sP)\right)[-1] =f^*\left({^p \cH}^{1} Rf_!(\sP)\right)[1].\ee
So if $^p \cH^{1} Rf_!(\sP)$ is non-zero then, since $\sP$ is simple, the epimorphism (\ref{epi}) is an isomorphism:
$$f^*\left( ^p \cH^{1} Rf_!(\sP)\right)[1]\cong \sP.$$
The desired cohomology vanishing follows in this case from Lemma \ref{pull}(i), by using the induction hypothesis applied to the perverse sheaf $^p \cH^{1} Rp_!(\sP)$ on $G'$. 
On the other hand, if $^p \cH^{1} Rf_!(\sP)$ is zero, then $Rf_!(\sP)$ is a perverse sheaf on $G'$. The desired vanishing then follows as in part $(i)$ by using Lemma \ref{iso}, the induction hypothesis for the perverse sheaf $Rf_!(\sP)$ on $G'$, and localization.
\end{proof}

\br
Our proof of the above Theorem \ref{semi} follows the strategy employed for proving 
Gabber-Loeser's Theorem \ref{GL}  for perverse sheaves on a complex affine torus $T$.  
In the case of an abelian variety $A$, Bhatt-Schnell-Scholze \cite{BSS} proved Theorem \ref{BSS} by using the projection formula and Artin's vanishing theorem on the universal cover of $A$, together with duality statements warranted by the compactness of $A$. In our context, i.e., working with perverse sheaves on a semi-abelian variety $G$, since the universal cover of $G$ is still a Stein space, we get by the projection formula and Artin's vanishing theorem on the universal cover of $G$ that $H^i(\cM_!(\sP))=0$ for all $i<0$. However, 
the lack of compactness of $G$ does not allow us to translate this vanishing into part $(i)$ of 
Theorem \ref{semi}. Instead, upon composing with the natural involution on $\Gamma_G$ and using the duality formula (\ref{dual}), this gives an alternative proof of part $(ii)$ in  Theorem \ref{semi}. We leave the details to the interested reader.
\er

\bc\label{satisfy_assumption}
Let $\sP$ be a $\kk$-perverse sheaf on $G$.  Any bounded complex $F^\ubul$ of finitely generated free $\Gamma_G$-modules representing $\sM_*(\sP)$ satisfies Assumption \ref{2vanishing}.
\ec
\begin{proof}
The assertion follows immediately from statements (i) and (ii) of Theorem \ref{semi}.
\end{proof}

\bc\label{mell} 
Let $\kk$ be a fixed field. If $\sF\in \,{^p D}^{\geq 0}(G,\kk)$, then $\sM_*(\sF) \in D^{\geq 0} (\Gamma_G)$. 
\ec
\begin{proof}
{Since $\sM_*$ is a functor of triangulated categories, the assertion follows from Theorem \ref{semi} and d\'evissage. }
\end{proof}

We can now prove the main result of this section.
\bt \label{package}   Let $\kk$ be a fixed field.  For any perverse sheaf $\sP\in \Perv(G,\kk)$, the cohomology jump loci of $\sP$ satisfy the following propagation package:
\begin{enumerate}
\item[(i)] {\it Propagation property}:
$$
 \sV^{-m-g}(G,\sP) \subseteq \cdots \subseteq \sV^{-1}(G,\sP) \subseteq \sV^{0}(G,\sP) \supseteq \sV^{1}(G,\sP) \supseteq \cdots \supseteq \sV^{g}(G,\sP).$$
Furthermore, $\sV^i(G,\sP)= \emptyset$ if $i\notin [-m-g, g]$.
\item[(ii)] {\it Codimension lower bound}: for any $i\geq 0$, $$ \codim \sV^{\pm i}(G,\sP) \geq i.$$
\end{enumerate} 
\et

\begin{proof}
We begin by showing that $\sV^i(G,\sP)=\emptyset$ for $i\notin [-m-g,g]$. Since the projection map $\pi: G\to A$ is an affine morphism of relative dimension $m$, by \cite[Corollary 5.2.14(ii), Theorem 5.2.16(i)]{D2}, we have
\begin{equation}\label{interval}
^p \cH^{\ell} R\pi_*(\sP)=0 \ \text{ if } \ \ell\notin [-m, 0].
\end{equation}
On the other hand, since $A$ is an irreducible algebraic variety of complex dimension $g$ and $^p \cH^{\ell} R\pi_*(\sP)$ is a perverse sheaf on $A$, we have 
\begin{equation}\label{interval2}
H^k(A, {^p \cH}^{\ell} R\pi_*(\sP))=0 \ \text{ if } \ k \notin [-g,g],
\end{equation}
 e.g., see \cite[Proposition 5.2.20]{D2}. By the perverse Leray spectral sequence for $\pi$, i.e., 
\be\label{pss}
E_2^{k,\ell}=H^k(A, {^p \cH}^{\ell} R\pi_*(\sP)) \Longrightarrow H^{k+\ell}(G,\sP),
\ee
and the vanishing properties (\ref{interval}) and (\ref{interval2}), 
it follows readily that $H^i(G,\sP)=0$ for $i \notin [-m-g,g]$. In order to show that $\sV^i(G,\sP)=\emptyset$ for $i\notin [-m-g,g]$, we apply the above reasoning to the $\kk_{\rho}$-perverse sheaf $\sP \otimes_{\kk} L_{\rho}$, with $L_{\rho}$ the rank-one local system associated to any $\rho \in \spec \Gamma_G$. 

We next notice that, by Remark \ref{projective_free}, we can represent $\sM_*(\sP)$ by a bounded complex of finitely generated free $\Gamma_{G}$-modules $F^\ubul$. By Corollary \ref{satisfy_assumption}, $F^\ubul$ satisfies Assumption \ref{2vanishing}. Now, the codimension lower bound $(ii)$ follows from Proposition \ref{prop} and Remark \ref{coinc} using the fact that $\Gamma_G$ is a Cohen-Macauley ring (hence the depth is same as the codimension).
\end{proof}

As an immediate consequence of Theorem \ref{package}, we get {a new proof of the following} known result:
\bc\label{gvcor} Let $\kk$ be a fixed field.  Any perverse sheaf $\sP\in \Perv(G,\kk)$ satisfies:
\begin{enumerate}
\item[(i)] {\it Generic vanishing}: there exists a non-empty Zariski open subset $U \subset \spec \Gamma_G $ such that,  for any maximal ideal $\rho\in U$,  $H^{i}(G, \sP\otimes_{\kk} L_\rho)=0$ for all $i\neq 0$.
\item[(ii)] {\it Signed Euler characteristic property}:  $$  \chi(G,\sP)\geq 0.$$
Moreover, the equality holds if and only if $\sV^0(G,\sP) \neq \spec \Gamma_G$.
\end{enumerate} 
\ec

\begin{proof}
The generic vanishing of $(i)$ is a direct consequence of Theorem \ref{package}. Indeed, by the codimension lower bound, it follows that $\bigcup_{i \neq 0} \sV^i(G,\sP)$ has at least codimension $1$ in $\spec \Gamma_G$. Moreover, the propagation property yields that if $\rho \notin \bigcup_{i \neq 0} \sV^i(G,\sP)$ then $H^{i}(G, \sP\otimes_{\kk} L_\rho)=0$ for all $i\neq 0$. For such a generic choice of $\rho \in \spec \Gamma_G $, we have that 
$$\chi(G,\sP)=\chi(G,\sP\otimes_{\kk} L_\rho)=\dim_{\kk} H^0(G, \sP\otimes_{\kk} L_\rho)\geq 0.$$
The last claim in $(ii)$ follows immediately from the propagation property of Theorem \ref{package}.
\end{proof}

\br
An equivalent formulation of the propagation property for cohomology jump loci of perverse sheaves (Theorem \ref{package}(i)) is the following. For $\sP\in \Perv(G, \kk)$, suppose that not all cohomology groups $H^j(G, \sP)$ are zero. Let 
$$k_+:=\max\{j|H^j(G, \sP)\neq 0\} \text{ and } k_-:=\min\{j|H^j(G, \sP)\neq 0\}.$$
Then Theorem \ref{package}(i) is equivalent to $k_+\geq 0$, $k_-\leq 0$ and
$$H^j(G, \sP)\neq 0 \ \iff \ k_-\leq j\leq k_+. $$
Moreover, if $\kk$ is a field of characteristic zero and if $\sP$ is a semi-simple perverse sheaf on an abelian variety $G=A$, then the Hard Lefschetz Theorem \ref{HL} yields that $k_-=-k_+$. Then by the relative Hard Lefschetz theorem for the Albanese map of a smooth projective variety, we also recover \cite[Corollary 1]{We16b}. 
\er

\section{Simple $\bC$-perverse sheaves with vanishing Euler number}\label{zeroeu}
In this section, we investigate simple perverse sheaves with Euler number zero.
Throughout this section, we fix $\kk=\bC$, thus $\spec \Gamma_G \cong (\bC^*)^{m+2g}$. 

\bd A {\it linear subvariety}  of $\spec \Gamma_G$ is a closed subvariety of $\spec \Gamma_G$ of the form:
\begin{center}
$\rho \cdot \im (f^\#: \spec \Gamma_{G'} \to \spec \Gamma_{G})$
\end{center}
where $f: G\to G'$ is a surjective homomorphism of semi-abelian varieties with connected fibers and $\rho \in \spec \Gamma_G$ is a rank-one character. 
\ed

\bd For any $\sF \in D^b_c(G,\bC)$, set
$$\sV^i(G, \sF):= \{\rho \in \spec \Gamma_G \mid H^i(G, \sF\otimes_{\bC} L_\rho)\neq 0\}.$$ 
\ed

The following important property for $\sV^i(G, \sF)$ follows from the proof of \cite[Theorem 10.1.1]{BW17}.
\bt[Structure Theorem] \label{linear} Let $G$ be a complex semi-abelian variety and let $\sF\in D^b_c(G,\bC)$ be a  bounded $\bC$-constructible complex on $G$. Then each $\sV^i(G,\sF)$ is a finite union of linear subvarieties of $\spec \Gamma_G $. \et

We also need the following fact, similar to \cite[Proposition 3.4.6]{GL}.
\begin{prop}\label{nonzero}
A constructible complex $\sF\in D^b_c(G, \bC)$ is the zero object if and only if $\sM_*(\sF)=0$. In particular, if $\sF\in D^b_c(G, \bC)$ is nonzero, then $\sV^i(G, \sF)\neq \emptyset$ for some $i\in \bZ$. 
\end{prop}
\begin{proof}
We can use the same argument as in the proof of \cite[Proposition 3.4.6]{GL}, together with our inductive scheme from Theorem \ref{semi}, to reduce the proof to the case when $G=A$ is an abelian variety. Then the assertion follows from the Riemann-Hilbert correspondence and the fact that the Fourier-Mukai transform is an equivalence of categories (see \cite{Sch}). 
\end{proof}

The following theorem provides a unification and generalization to the semi-abelian context of similar statements from the abelian case (cf. \cite[Theorem 2]{We0}, \cite[Proposition 10.1(a)]{KW}, \cite[Theorem 7.6]{Sch}) and from the affine torus case (cf. \cite[Theorem 5.1.1]{GL}), respectively.
\bt \label{simple} Let $G$ be a complex semi-abelian variety.   If $\sP \in \Perv(G, \bC)$ is a simple perverse sheaf on $G$ with $\chi(G, \sP)=0$, then there exists a positive dimensional semi-abelian subvariety $G^{\prime \prime}$ of $G$, a rank-one $\bC$-local system $L_\rho$ on $G$ and a simple perverse sheaf $ \sP^\prime$
on $G^\prime=G/G^{\prime \prime}$ with $\chi(G', \sP')\neq 0$, such that  
\be\label{pul} \sP\cong L_\rho \otimes_{\bC} f^* \sP^\prime[\dim G^{\prime \prime}],\ee with $f: G\to G'=G/G^{\prime \prime}$ denoting the quotient map.
Moreover, 
$$\sV^0(G,\sP)=\rho^{-1}\cdot \im (f^\#: \spec \Gamma_{G'} \to \spec \Gamma_{G})$$
is an irreducible linear subvariety. 
\et

\begin{proof}
Since $\chi(G, \sP)=0$, Corollary \ref{gvcor}(ii) yields that $\sV^0(G, \sP)\neq \spec \Gamma_G$. 
Proposition \ref{nonzero} and the propagation property in Theorem \ref{package} show that $\sV^0(G,\sP)$ is non-empty.   
Assume that $\sV^0(G,\sP)$ has codimension $d$, and let $V$ be an irreducible component of $\sV^0(G,\sP)$ of codimension exactly $d$. 
By Theorem \ref{linear},  $V$ is a linear variety. Without loss of generality, (after a suitable twist) we may assume that $V$ contains the constant sheaf. 
Then there exists a map of algebraic groups $f:G\to G'$ from $G$ to a semi-abelian variety $G'$
 such that $$V=f^\#(\spec \Gamma_{G'}).$$ Assume that the semi-abelian variety $G^{\prime \prime}=\ker(f)$ has the affine torus part $T''$ of dimension $m^{\prime\prime}$ and the abelian variety part $A''$ of dimension $g^{\prime\prime}$.  Then $d= m^{\prime\prime}+2g^{\prime\prime}$ and $\dim G^{\prime\prime}=m^{\prime\prime}+g^{\prime\prime}$. 
 
Write $f$ as a composition of an affine map $f_1: G\to G/T''$ of relative dimension $m''$ and a proper map $f_2: G/T''\to G/G''$ of relative dimension $g''$. By \cite[Corollary 5.2.14(ii) and Theorem 5.2.16(i)]{D2}, we have 
$$Rf_{1*}(\sP)\in\,{^p D}^{\geq -m''}(G/T'', \bC)\cap\, {^p D}^{\leq 0}(G/T'', \bC).$$
By \cite[Corollary 5.2.14]{D2}, using the properness of $f_2$, the $t$-amplitude of $Rf_{2*}$ is $[-g'', g'']$. Hence, we have
$$Rf_{*}(\sP)\in\,^p D^{\geq -m''-g''}(G', \bC)\cap\, ^p D^{\leq g''}(G', \bC).$$
Thus, the only possibly non-trivial perverse cohomology sheaves $^p \cH^i Rf_{*}(\sP)$  appears in the range $i\in [-m^{\prime\prime}-g^{\prime\prime}, g^{\prime\prime}]$. 
By the perverse hypercohomology spectral sequence, we have that 
$$\sV^{i}(G',Rf_*\sP) \subset  \bigcup_{j =-m^{\prime\prime}-g^{\prime\prime}}^{g^{\prime\prime}}\sV^{i-j}(G', \text{} ^p \cH^{j} Rf_*\sP).$$
By Theorem \ref{package}$(ii)$, this yields $\sV^i(G',Rf_*\sP)\neq \spec  \Gamma_{G'}$ for $i\notin [-m^{\prime\prime}-g^{\prime\prime}, g^{\prime\prime}]$. 
    
 On the other hand, for any $\bC$-coefficient local system $L$ on $G'$, we have the projection formula $$Rf_* \sP \otimes_\bC L \cong Rf_* (\sP \otimes_\bC f^* L),$$ which implies that $$f^\# (\sV^i(G',Rf_*\sP))= V \cap \sV^i(G,\sP).$$
By Theorem \ref{semi} and Proposition \ref{prop}$(iii)$, there exist integers $a_1$ and $a_2$ with $\sV^i(G',Rf_*\sP)= \spec \Gamma_{G'}$ for $i\in [a_1,a_2]$. So the interval $[-m^{\prime\prime}-g^{\prime\prime}, g^{\prime\prime}]$ contains the interval $[a_1,a_2]$. However, since the interval $[-m^{\prime\prime}-g^{\prime\prime}, g^{\prime\prime}]$ has length $d$ and $[a_1,a_2]$ has length at least $d$, we get $$[-m^{\prime\prime}-g^{\prime\prime}, g^{\prime\prime}]=[a_1,a_2].$$ 
 
 We next note that $\sV^{-m^{\prime\prime}-g^{\prime\prime}}(G',Rf_*\sP) =\spec \Gamma_{G'}$ if and only if
 \begin{center}
 $ ^p \cH^{-m^{\prime\prime}-g^{\prime\prime}} Rf_*(\sP)\neq 0$ \ and \ $\chi(G',\text{} ^p \cH^{-m^{\prime\prime}-g^{\prime\prime}} Rf_*(\sP)) \neq 0$.
\end{center}  
Indeed, by the perverse hypercohomology spectral sequence, we have that $$\sV^{-m^{\prime\prime}-g^{\prime\prime}}(G',Rf_*\sP) \subset  \bigcup_{i =-m^{\prime\prime}-g^{\prime\prime}}^{g^{\prime\prime}}\sV^{-m^{\prime\prime}-g^{\prime\prime}-i}(G', \text{} ^p \cH^{i} Rf_*\sP).$$
Moreover, by Theorem \ref{package}$(ii)$,  $\sV^{-m^{\prime\prime}-g^{\prime\prime}-i}(G', \text{} ^p \cH^{i} Rf_*\sP)$ has codimension at least 1 in $\spec \Gamma_{G'}$ for $i>-m^{\prime\prime}-g^{\prime\prime}$. Hence, for $\sV^{-m^{\prime\prime}-g^{\prime\prime}}(G',Rf_*\sP) =\spec \Gamma_{G'}$ the only possibility is that  $\sV^{0}(G', \text{} ^p \cH^{-m^{\prime\prime}-g^{\prime\prime}} Rf_*\sP)= \spec \Gamma_{G'}$. By Corollary \ref{gvcor}$(ii)$, this is equivalent to the fact that $\chi(G',\text{} ^p \cH^{-m^{\prime\prime}-g^{\prime\prime}} Rf_*(\sP)) \neq 0$.

Since $f$ is a smooth map of relative dimension $m^{\prime\prime}+g^{\prime\prime}$, 
 it follows from \cite[Page 111]{BBD} that there is a canonical monomorphism of perverse sheaves
\be\label{mn3} f^*\left( ^p \cH^{-m^{\prime\prime}-g^{\prime\prime}} Rf_*(\sP)\right)[m^{\prime\prime}+g^{\prime\prime}]\hookrightarrow \sP.\ee
Since $\sP$ is simple, the monomorphism (\ref{mn3}) is an isomorphism.  So $$\sP':=\,^p \cH^{-m^{\prime\prime}-g^{\prime\prime}} Rf_*(\sP) \in \Perv(G',\bC)$$ satisfies all properties required in the theorem.

For any rank-one $\bC$-local system $L \notin V$, since $L\vert_{f^{-1}(y)}$ is  a non-trivial local system for any point $y\in G'$, we have $(Rf_* L)_y=0$ for any point $y\in G'$. Hence $Rf_* L=0$. Therefore, 
 $$ Rf_* (\sP\otimes_\bC L)=  Rf_* (f^* \sP^\prime[\dim G^{\prime \prime}]\otimes _\bC L)= \sP^\prime[\dim G^{\prime \prime}] \otimes_\bC  Rf_* L=0.$$
Hence $H^i(G, \sP\otimes _\bC L)= H^i(G', Rf_* (\sP\otimes _\bC L))=0$ for all $i$. This shows that $\sV^0(G, \sP)= V$, and hence irreducible.
\end{proof}

\bc \label{survive} If  $V$ is an irreducible component of $\sV^0(G,\sP)$ of codimension $d$, then $V$ is also an irreducible component of $\sV^{i}(G,\sP)$ exactly for $i\in [-m^{\prime\prime}-g^{\prime\prime}, g^{\prime\prime}]$, where $m^{\prime\prime}$ and $g^{\prime\prime} $ are the dimensions of the affine part and, resp., the abelian part of $G^{\prime\prime}$, as introduced in Theorem \ref{simple}. In particular, $d=m^{\prime\prime}+2g^{\prime\prime}$.
\ec
\begin{proof}
Indeed, in the proof of Theorem \ref{simple},  one gets the equality
 $$[-m^{\prime\prime}-g^{\prime\prime}, g^{\prime\prime}]=[a_1,a_2]$$
 without making use of the fact that the perverse sheaf $\sP$ was assumed to be simple.  So the assertion holds for any perverse sheaf.
\end{proof}

\br \label{codim} Assume that $G=A$ is an abelian variety.   If  $V$ is an irreducible component of $\sV^0(A,\sP)$ of codimension $d$, the structure theorem shows that $d$ has to be even and $g^{\prime\prime}=\frac{d}{2}$. So $V$ is  an irreducible component of $\sV^{i}(A,\sP)$ exactly for $i\in [-\dfrac{d}{2}, \dfrac{d}{2}]$. 
Moreover, if $\sV^0(A,\sP)$ has codimension $d> 0$, then $$\sV^0(A,\sP)= \sV^{\pm 1}(G,\sP)= \cdots = \sV^{\pm d/2}(A,\sP) \neq \sV^{\pm (1+d/2)}(A,\sP).$$ 
 Similar results for the case of a complex affine torus have been obtained in \cite[Theorem 1.2$(iv)$]{LMW3}.
\er

\section{Characterization of $\bC$-perverse sheaves} \label{cpv}
In this section, we use the Structure Theorem \ref{linear}  
and more refined codimension lower bounds for cohomology jumping loci to give a complete characterization of $\bC$-perverse sheaves on semi-abelian varieties. Results in this section generalize the corresponding results for perverse sheaves on abelian varieties obtained by Schnell (cf. \cite[Theorem 7.4]{Sch}), as well as the results of Gabber-Loeser for perverse sheaves on complex affine tori (cf. \cite[Theorem 3.4.7]{GL}).

Throughout this section we fix $\kk=\bC$. Let $G$ be as before a complex semi-abelian variety of dimension $m+g$, with $\Gamma_G:=\bC[\pi_1(G)]\cong \bC[t_1^{\pm 1}, \cdots, t_{m+2g}^{\pm 1}].$ Here $g$ denotes the complex dimension of the abelian part, and $m$ is the dimension of the affine torus part.

Let $\sF\in D^b_c(G,\bC)$ be a  bounded constructible complex of $\bC$-sheaves on $G$, and let $V$ be an irreducible component of $\sV^i(G,\sF)$. By the Structure Theorem \ref{linear}, we know that $V$ is linear, and hence there is a surjective homomorphism $f(V):G \to G'(V)$ of semi-abelian varieties such that, up to a translate, $V$ is equal to the image of 
$$f(V)^\#: \spec \Gamma_{G'(V)} \to \spec \Gamma_{G}.$$
Let 
$$1\to T'(V)\to G'(V)\to A'(V)\to 1$$
be the group extension corresponding to $G'(V)$, with $T'(V)$ a complex affine torus and $A'(V)$ a complex abelian variety. 
Let $G''(V):=\ker f(V)$, with $T''(V)$ and $A''(V)$ denoting the affine torus and, resp., the  abelian variety part of $G''(V)$.
To formulate the results of this section, we introduce the following terminology. 
\bd \label{GA}
In the above notations, we define the \emph{semi-abelian dimension} of $\sV^i(G,\sF)$ by 
$$\dim_{sa}\sV^i(G, \sF)=\max_{V}\{\dim G'(V)\}$$
and the \emph{semi-abelian codimension} of $\sV^i(G,\sF)$ by 
$$\codim_{sa} \sV^i(G,\sF) =\dim G-\dim_{sa}\sV^i(G,\sF)= \min\limits_{V} \{\dim G''(V)\},$$
where $V$ runs over all irreducible components of $ \sV^i(G,\sF)$.

Similarly, we define the \emph{abelian dimension} of $\sV^i(G, \sF)$ by 
$$\dim_{a}\sV^i(G, \sF)=\max_{V}\{\dim A'(V)\},$$
and its \emph{abelian codimension} by 
$$\codim_{a} \sV^i(G,\sF) = \dim A-\dim_{a} \sV^i(G,\sF)=\min\limits_{V} \{\dim A''(V)\},$$
where $V$ runs over all irreducible components of $ \sV^i(G,\sF)$.
\ed

\br
Let $V$ be a nonempty linear subvariety of $\spec \Gamma_G$. 
\begin{enumerate}
\item If $G=T$ is a complex affine torus, then $\dim_{sa}(V)=\dim(V)$, $\codim_{sa}(V)=\codim(V)$, and $\dim_{a}(V)=\codim_{a}(V)=0$, $\dim_{a}(\emptyset)=-\infty$, $\codim_a(\emptyset)=\infty$. 
\item If $G=A$ is a complex abelian variety, then $\dim_{sa}(V)=\dim_{a}(V)=\frac{1}{2}\dim(V)$, $\codim_{sa}(V)=\codim_{a}(V)=\frac{1}{2}\codim(V)$.
\end{enumerate}
\er

The first result of this section provides more refined codimension lower bounds for the cohomology jump loci of $\bC$-perverse sheaves on $G$, as follows.

\bp\label{bdeq}
 Let $\sP$ be a $\bC$-perverse sheaf on the semi-abelian variety $G$.
\begin{enumerate}
\item
For any $   i\geq 0$, we have the following abelian codimension bound:
$$\codim_{a} \sV^i(G, \sP) \geq i.$$ 
Moreover, there exist $i_+\geq 0$, such that $\codim_{a} \sV^{i_+}(G, \sP) = i_+.$
\item  
For any $i\leq 0$, we have the following semi-abelian codimension bound:
$$\codim_{sa} \sV^i(G,\sP)\geq -i.$$
Moreover, there exist $i_-\leq 0$, such that $\codim_{sa} \sV^{i_-}(G, \sP) = i_-.$
\end{enumerate}
\ep
\begin{proof}
We first prove the existence of $i_+$ and $i_-$. If $\sV^0(G, \sP)=\spec \Gamma_G$ is the whole moduli space, one can take $i_+=i_-=0$. Otherwise, let $V $ be an irreducible component of $\sV^0(G,\sP)$. One can take $i_+=\dim A''(V)$ and $i_-=-\dim G''(V)$, with $A''(V)$ and $G''(V)$ as in Definition \ref{GA}. 
Then the desired equalities follow from Corollary \ref{survive}. 

Next, we prove the codimension lower bounds. 
By induction on the length of $\bC$-perverse sheaves, we may assume that $\sP$ is simple since the conclusion behaves well under exact sequences.

We prove the codimension bounds by induction on the dimension of the torus $T$.
If $\dim T=0$, then $G$ is an abelian variety, and the assertions are equivalent to Theorem \ref{Schnell}. For the induction step we proceed as follows.

Let  $f:G \to G'$ be the projection map defined as before, by forgetting the first coordinate of the affine torus $T$. Set $$V_{\lambda}:=\{t\in \spec \Gamma_G \mid t_1=\lambda\},$$ which is a codimension-one subtorus of $\spec \Gamma_G$, and note that 
\be\label{uni}\bigcup_{\lambda\in \bC^*} V_\lambda= \spec \Gamma_G.
\ee 
The map $f:G\to G'$ induces an embedding on the moduli spaces: $$f^\#: \spec \Gamma_{G'} \to \spec \Gamma_{G},$$
whose image coincides with $V_1$, i.e.,  $f^\#(\spec \Gamma_{G'})\cong V_1$.

If $^p \cH^{-1} Rf_*(\sP)\neq 0$, then we get as in the proof of Theorem \ref{semi}$(i)$ an isomorphism:
$$f^*\left( ^p \cH^{-1} Rf_*(\sP)\right)[1]\cong \sP.$$
In other words, fixing a splitting $G\cong G'\times \cc^*$ as topological spaces, we can express $\sP$ as an external product 
$$\sP\cong\, ^p \cH^{-1} Rf_*(\sP)\boxtimes \bC_{\cc^*}[1].$$
Then the K\"unneth formula (see, e.g., \cite[Theorem 4.3.14]{D2}) yields that 
$$\sV^i(G, \sP)=\bigcup_{k}\sV^{i-k}(G', {^p \cH}^{-1} Rf_*(\sP))\times \sV^{k}(\cc^*, \cc_{\cc^*}[1]),$$
where we identify $\spec \Gamma_G$ with $\spec \Gamma_{G'}\times \spec \Gamma_{\cc^*}$ by the natural isomorphism induced by $\pi_1(G)\cong \pi_1(G')\times \pi_1(\cc^*)$. 
Since 
$$\sV^{k}(\cc^*, \bC_{\cc^*}[1])=
\begin{cases}
\emptyset &\text{ if } k\neq -1, 0,\\
\{\mathbf{1}\} &\text{ if } k=-1, 0,
\end{cases}
$$
we have
\be\label{cod}
\sV^i(G,\sP)= f^\#(\sV^{i}(G', {^p \cH}^{-1} Rf_*(\sP))) \cup f^\# ( \sV^{i+1}(G', {^p \cH}^{-1} Rf_*(\sP))).
\ee
In particular, in this case, $\sV^i(G,\sP) \subset V_1$. Since $G$ and $G'$ have the same abelian part, for any $i\geq 0$ we have that 
\begin{multline}\label{cod0}
\codim_{a} \sV^i(G,\sP) = \\ \min \{ \codim_{a} \sV^{i}(G', {^p \cH}^{-1} Rf_*(\sP)) , \codim_{a} \sV^{i+1}(G', {^p \cH}^{-1} Rf_*(\sP))\}.
\end{multline}
On the other hand, by the induction hypothesis for the perverse sheaf $^p \cH^{-1} Rf_*(\sP)$ on $G'$, for any $i\geq 0$ we have 
\be\label{cod1} \codim_{a} \sV^{i}(G', {^p \cH}^{-1} Rf_*(\sP))\geq i,
\ee
and
\be\label{cod2}\codim_{a} \sV^{i+1}(G', {^p \cH}^{-1} Rf_*(\sP))\geq i+1.
\ee
Therefore, by (\ref{cod0}), (\ref{cod1}) and (\ref{cod2}), we get that for any $i\geq 0$, 
$$
\codim_{a}\sV^i(G,\sP)  \geq i.
$$

Similarly, by (\ref{cod}), we have
\begin{equation}\label{cod3}
\dim_{sa} \sV^i(G,\sP) = \max \{ \dim_{sa} \sV^{i}(G', {^p \cH}^{-1} Rf_*(\sP)) , \dim_{sa} \sV^{i+1}(G', {^p \cH}^{-1} Rf_*(\sP))\}.
\end{equation}
By the induction hypothesis for the perverse sheaf $^p \cH^{-1} Rf_*(\sP)$ on $G'$, for any $i\leq 0$ we have that 
\be\label{cod4} \dim_{sa} \sV^{i}(G', {^p \cH}^{-1} Rf_*(\sP))\leq \dim G'+ i,
\ee
and
\be\label{cod5}\dim_{sa} \sV^{i+1}(G', {^p \cH}^{-1} Rf_*(\sP))\leq \dim G'+i+1=\dim G+i.
\ee
By (\ref{cod3}), (\ref{cod4}) and (\ref{cod5}), we get that for any $i\leq 0$, 
$$
\dim_{sa}\sV^i(G,\sP) \leq \dim G +i,
$$
and hence
$$
\codim_{sa}\sV^i(G,\sP) \geq -i.
$$

So far, we proved the codimension lower bounds assuming $^p \cH^{-1} Rf_*(\sP)\neq 0$. 
More generally, if there exists a rank-one $\bC$-local system $L$ on $G$, such that $^p \cH^{-1} Rf_*(\sP\otimes_\bC L)\neq 0$, then one obtains the same codimension lower bounds. In fact, tensoring with $L$ induces a translation on the cohomology jump loci, and in particular preserves the (semi-)abelian codimensions. 

Assume next that there is no rank-one local system $L$ on $G$ satisfying the condition $^p \cH^{-1} Rf_*(\sP\otimes_\bC L)\neq 0$. For any $\lambda\in \bC^*$, let us choose a rank-one local system $L_\lambda$ whose corresponding point in $\spec \Gamma_G$ is contained in $V_\lambda$. 

By the above assumption, $^p \cH^{-1} Rf_*(\sP\otimes_\bC L_\lambda)=0$, and hence $Rf_* (\sP\otimes_\bC L_\lambda)$ is a perverse sheaf on $G'$. By the projection formula, we have that
$$\sV^i(G,\sP\otimes_\bC L_\lambda)\cap V_1 = f^\#(\sV^i(G', Rf_* (\sP\otimes_\bC L_\lambda))),$$
or, equivalently,
$$\sV^i(G,\sP)\cap V_{\lambda^{-1}} =\lambda^{-1}\cdot f^\#(\sV^i(G', Rf_* (\sP\otimes_\bC L_\lambda))).$$
By the induction hypothesis for the perverse sheaf $Rf_* (\sP\otimes_\bC L_\lambda)$ on $G'$, we have that
\begin{equation}\label{eq_a}
\codim_{a} (\sV^i(G,\sP)\cap V_{\lambda^{-1}}) = \codim_{a} \sV^i (G', Rf_*( \sP\otimes_\bC L_\lambda)) \geq i
\end{equation}
for any $i\geq 0$, and
\begin{equation}\label{eq_sa}
  \dim_{sa} (\sV^i(G,\sP)\cap V_{\lambda^{-1}}) = \dim_{sa} \sV^i (G', Rf_* (\sP\otimes_\bC L_\lambda)) \leq \dim G'+i
\end{equation}
for any $i\leq 0$. 
We can interpret (\ref{eq_sa}) as saying that the semi-abelian codimension of $\sV^i(G,\sP)\cap V_{\lambda^{-1}}$ within $V_{\lambda^{-1}}$ is no less than $i$. Since (\ref{eq_a}) holds for any $\lambda\in \bC^*$, by (\ref{uni}), we have
$$\codim_{a} (\sV^i(G,\sP))\geq i$$
for any $i\geq 0$. Similarly, since (\ref{eq_sa}) holds for any $\lambda\in \bC^*$, by (\ref{uni}), we get that 
$$ \codim_{sa} (\sV^i(G,\sP))\geq -i$$
for any $i\leq 0$. 
\end{proof}

\br If $\sP$ is a simple $\bC$-perverse sheaf on $G$ with $\chi(G,\sP) \neq 0$, then it can be shown by using methods similar to those in Theorem \ref{simple} that the following stronger codimension bounds hold:
$$\codim_{a} \sV^i(G, \sP) \geq i+1 , \ {\rm for} \ i >0,$$ 
and 
$$\codim_{sa} \sV^i(G,\sP)\geq -i+1  , \ {\rm for} \ i <0.$$
This generalizes a fact obtained by Schnell in the abelian context, cf. \cite[Section 5]{Sch}.
\er

\begin{cor}\label{inequal}
Let $\sF\in D^b_c(G, \bC)$ be a bounded $\bC$-constructible complex on $G$. 
\begin{enumerate}
\item If $\sF\in \,^p D^{\leq 0}(G, \bC)$, that is, ${^p\sH}^j(\sF)=0$ for any $j>0$, then for any $i\geq 0$,
$$\codim_{a} \sV^i(G, \sF) \geq i.$$
\item If $\sF\in \,^p D^{\geq 0}(G, \bC)$, that is, ${^p\sH}^j(\sF)=0$ for any $j<0$, then for any $i\leq 0$,
$$\codim_{sa} \sV^i(G, \sF) \geq -i.$$
\end{enumerate}
\end{cor}

\begin{proof} 
Given any rank-one $\bC$-local system $L$ on $G$, notice that
$$^p\sH^j(\sF\otimes_\bC L)=\,^p\sH^j(\sF)\otimes_\bC L. $$
Therefore, by the perverse cohomology spectral sequence
$$E_2^{i-j, j}=H^{i-j}(G, \,^p\sH^j(\sF)\otimes_\bC L)\Rightarrow H^{i}(G, \sF\otimes_\bC L),$$
we get an inclusion
\be\label{co1}
\sV^i(G, \sF) \subset \bigcup_j \sV^{i-j}(G, {^p\sH}^j(\sF)).
\ee
If $\sF\in \,^p D^{\leq 0}(G, \bC)$, then
$$\codim_{a} \sV^i(G, \sF)   \geq \min_{j\leq 0} \codim_{a} \sV^{i-j}(G, {^p\sH}^j(\sF)) \geq \min_{j\leq 0} (i-j) \geq i,$$
where the second inequality follows from Proposition \ref{bdeq} (1). Similarly, if $\sF\in \,^p D^{\geq 0}(G, \bC)$, Proposition \ref{bdeq} (2) yields that
\[
\codim_{sa} \sV^i(G, \sF)   \geq \min_{j\geq 0} \codim_{sa} \sV^{i-j}(G, {^p\sH}^j(\sF)) \geq \min_{j\geq 0} (j-i) \geq -i.\qedhere
\]
\end{proof}

The main result of this section provides a complete description of $\bC$-perverse sheaves on a semi-abelian variety $G$ in terms of their cohomology jump loci, as follows. 
\begin{theorem} \label{characterization}
Let $\sF\in D^b_c(G, \bC)$ be a bounded $\bC$-constructible complex on $G$. We have:
\begin{enumerate}
\item[(a)] $\sF\in \,^p D^{\leq 0}(G, \bC) \ \iff \ \codim_{a} \sV^i(G, \sF) \geq i \text{ for any } i\geq 0,$
\item[(b)] $\sF\in \,^p D^{\geq 0}(G, \bC) \ \iff \ \codim_{sa} \sV^i(G, \sF) \geq -i \text{ for any } i\leq 0.$
\end{enumerate}
Thus, $\sF$ is a $\bC$-perverse sheaf on $G$ if and only if the following two conditions are satisfied:
\begin{enumerate}
\item $\codim_{a} \sV^i(G, \sF) \geq i$ for any $i\geq 0$,
\item $\codim_{sa} \sV^i(G, \sF) \geq -i$ for any $i\leq 0$. 
\end{enumerate}
\end{theorem}

\begin{proof}
$(a)$ The implication $\Rightarrow$ follows from Corollary \ref{inequal}. For the converse, let $\sF\in D^b_c(G,\bC)$ be a  bounded constructible complex satisfying the abelian codimension bound in nonnegative degrees $i\geq 0$. Suppose $\sF\notin \,^p D^{\leq 0}(G, \bC)$. 

Let $i_0$ be the largest positive integer such that $^p \sH^{i_0}(\sF)\neq 0$. 
Then there is a morphism in $D^b_c(G,\bC)$,
$$\sF\to \,^p \sH^{i_0}(\sF)[-i_0]$$
inducing an isomorphism on the $i_0$-th perverse cohomology. {Let $\sF_0[1]$ be the mapping cone of the above morphism. Then we have the distinguished triangle in $D^b_c(G,\bC)$:}
$$\sF\to \,^p \sH^{i_0}(\sF)[-i_0]\to \sF_0[1]\xrightarrow[]{+1}.$$
By tensoring with any rank-one local system $L$, and considering the associated cohomology long exact sequence, we get the following inclusion of cohomology jump loci
\begin{equation}\label{eq_0}
\sV^{k-i_0}(G, \,^p \sH^{i_0}(\sF))\subset \sV^{k}(G, \sF)\cup \sV^{k+1}(G, \sF_0),
\end{equation}
for any $k\geq i_0$. 
Since $\sF_0=\,^p\tau^{<i_0}\sF_0$, we have by Corollary \ref{inequal}, after a shift of degree, that 
\begin{equation}\label{eq_1}
\codim_{a} \sV^{k+1}(G, \sF_0) \geq k-i_0+2,
\end{equation}
for any   $k\geq i_0$. Since $^p \sH^{i_0}(\sF)$ is a nonzero perverse sheaf, by Proposition \ref{bdeq} (1), there exists $k_0\geq i_0$ such that 
$$\codim_{a} \sV^{k_0-i_0}(G, \,^p \sH^{i_0}(\sF)) = k_0-i_0.$$
Plugging in $k=k_0$ in (\ref{eq_0}) and (\ref{eq_1}), we get that $\codim_{a} \sV^{k_0}(G, \sF)\leq k_0-i_0<k_0$, thus contradicting the hypothesis.

$(b)$ The implication $\Rightarrow$ follows from Corollary \ref{inequal}. For the converse, let $\sF\in D^b_c(G,\bC)$ be a  bounded constructible complex satisfying the semi-abelian codimension bound in nonpositive degrees $i\leq 0$. Suppose $\sF\notin \,^p D^{\geq 0}(G, \bC)$. 

Let $i_1$ be the smallest negative integer such that $^p \sH^{i_1}(\sF)\neq 0$. Then there is a morphism in $D^b_c(G,\bC)$,
$$^p \sH^{i_1}(\sF)[-i_1]\to \sF$$
inducing an isomorphism on the $i_1$-th perverse cohomology. {Let $\sF_1$ be the mapping cone of the above morphism. By turning the resulting triangle, 
we get a distinguished triangle in $D^b_c(G,\bC)$:}
\begin{equation}
\sF_1[-1]\to \,^p \sH^{i_1}(\sF)[-i_1]\to\sF\xrightarrow[]{+1}.
\end{equation}
By using the same argument as above, we have an inclusion
\begin{equation}\label{eq_3}
\sV^{k-i_1}(G, \,^p \sH^{i_1}(\sF))\subset \sV^{k}(G, \sF)\cup \sV^{k-1}(G, \sF_1),
\end{equation}
and inequality
\begin{equation}\label{eq_2bis}\codim_{sa} \sV^{k-1}(G,\sF_1)\geq -k+i_1+2,\end{equation}
for any $k\leq i_1$. Since $^p \sH^{i_1}(\sF)$ is a nonzero perverse sheaf, by Proposition \ref{bdeq} (2), there exists $k_1\leq i_1$ such that 
$$\codim_{sa} \sV^{k_1-i_1}(G, \,^p \sH^{i_1}(\sF)) = i_1-k_1.$$
Plugging in $k=k_1$ in (\ref{eq_3}) and (\ref{eq_2bis}), we get that $\codim_{sa} \sV^{k_1}(G, \sF) \leq  i_1-k_1<-k_1$. This contradicts the hypothesis. 
\end{proof}

\br\label{tstructure}
In general, one can not construct a $t$-structure on $D^b_{coh}(\Gamma_G)$ as in \cite{AB} such that the Mellin transformation $\sM_*: D^b_c(G, \bC)\to D^b_{coh}(\Gamma_G)$ is $t$-exact. In fact, let $A$ be an abelian variety of dimension two, let $T=(\cc^*)^4$ be the affine torus and let $G=T\times A$ be the splitting semi-abelian variety. Pulling back via the projections, we can consider $\spec \Gamma_A$ and $\spec \Gamma_T$ as linear subvarieties of $\spec \Gamma_G$. According to \cite{AB} (see also \cite{Kas}), the existence of such a $t$-structure on $D^b_{coh}(\Gamma_G)$ is equivalent to the existence of a  monotone and comonotone perversity function $p$ extending $-\dim_a$ on $\spec\Gamma_G$ (regarded as spectrum instead of maximum spectrum). By using Bertini's theorem repeatedly, we can construct a $5$-dimensional irreducible subvariety $Z$ containing both $\spec \Gamma_T$ and $\spec \Gamma_A$. Since $\dim_a \spec\Gamma_A=2$ and $\dim_a\spec\Gamma_T=0$, the monotonicity of $p$ implies that 
$$p(x_Z)\leq p(x_{\spec\Gamma_A})= -\dim_a \spec\Gamma_A=-2$$
and the comonotonicity of $p$ implies that 
$$p(x_Z)\geq p(x_{\spec\Gamma_T})-1= -\dim_a \spec\Gamma_T-1=-1$$
where $x_Z$, $x_{\spec\Gamma_A}$ and $x_{\spec\Gamma_T}$ are the generic points of $Z$, $\spec\Gamma_A$ and $\spec\Gamma_T$ as subschemes of $\spec\Gamma_G$, respectively. Such a perversity function can not exist. 
\er

The following  consequence of Theorem \ref{characterization} complements the results of Gabber-Loeser \cite{GL} (see Theorem \ref{GL}) on characterization of perverse sheaves on complex affine tori. 
\bc\label{specialize}
Suppose $G=T$ is a complex affine torus. A constructible complex $\sF\in D^b_c(T, \bC)$ is perverse on $T$ if and only if the following conditions hold.
\begin{enumerate}
\item For any $i>0$, $\sV^i(T, \sF)=\emptyset$. 
\item For any $i\leq 0$, $\codim \sV^i(T, \sF)\geq -i$. 
\end{enumerate}
\ec

\br
Item (1) of Corollary \ref{specialize} is equivalent to Artin's vanishing theorem for perverse sheaves on $T$ (see, e.g., \cite[Corollary 5.2.18]{D2}). 
\er

\bc \label{union}  
Let $0\to\sP' \to \sP \to \sP''\to 0 $ be a  short exact sequence of  perverse sheaves in $ \Perv(G,\bC)$. Then 
$$ \sV^0(G,\sP)= \sV^0(G,\sP')\cup \sV^0(G, \sP''). $$
\ec
\begin{proof} By tensoring the given short exact sequence of perverse sheaves with any rank one local system $L$, and considering the associated cohomology long exact sequence, we have the following inclusion of cohomology jump loci
$$ \sV^0(G,\sP)\subset \sV^0(G,\sP')\cup \sV^0(G, \sP'') .$$

On the other hand, let $V$ be an irreducible component of $ \sV^0(G,\sP')\cup \sV^0(G, \sP'')$ such that $V\not\subset \sV^0(G,\sP)$. Suppose that $V$ is an irreducible component of $\sV^0(G,\sP') $.  
By Proposition \ref{bdeq} (2), there exists $i_1\leq 0$ such that  \be\label{zer} \codim_{sa} \sV^{i_1}(G, \sP') = -i_1.\ee
After turning once the distinguished triangle associated to the given short exact sequence of perverse sheaves, 
we have that
\be\label{zer9} \sV^{i_1} (G, \sP') \subset  \sV^{i_1} (G, \sP) \cup  \sV^{i_1-1} (G, \sP'').\ee
By Proposition \ref{bdeq}, we get that $\codim_{sa} \sV^{i_1-1}(G, \sP'')\geq 1-i_1>-i_1$, thus by (\ref{zer}) and (\ref{zer9}) we have that $ V \subset \sV^{i_1} (G, \sP)$. The propagation property of Theorem \ref{package} then yields that $ V \subset \sV^{i_1} (G, \sP) \subset \sV^{0} (G, \sP)$, which contradicts our hypothesis.  
If $V$ is an irreducible component of $\sV^0(G,\sP'') $, the claim follows in a similar way  by using abelian codimension bound of Proposition \ref{bdeq}.  
\end{proof}

We conclude this section with the following result, which will be needed in the applications discussed in the next section. 

\bp \label{Alexander} Let $\sF\in D^b_c(G, \bC)$ be a $\bC$-constructible complex on $G$. Then we have the  equality
$$ \bigcup_i \sV^i(G,\sF) =\bigcup_j \sV^0(G, \,^p\sH^j(\sF)).$$
\ep
\begin{proof}
As in the proof of Corollary \ref{inequal}, the perverse cohomology spectral sequence
$$E_2^{i-j, j}=H^{i-j}(G, \,^p\sH^j(\sF)\otimes_\bC L)\Rightarrow H^{i}(G, \sF\otimes_\bC L),$$
yields an inclusion
 $$\bigcup_i \sV^i(G,\sF) \subset \bigcup_i \left(\bigcup_j \sV^{i-j} (G, \,^p\sH^j(\sF))\right)= \bigcup_j \sV^0(G, \,^p\sH^j(\sF)),$$ where the last equality follows from the propagation property of Theorem \ref{package}.

On the other hand, let $V$ be an irreducible component of $\bigcup_j \sV^0(G, \,^p\sH^j(\sF)) $.  Let $j_0 \in \zz$ be largest such that $V$ is  an irreducible component of $\sV^0(G, \,^p \sH^{j_0}(\sF))$.  Let $k=\codim_a V$ be the abelian codimension of $V$.  By Corollary \ref{survive}, we have
$$V\subset \sV^{k}(G, \,^p \sH^{j_0}(\sF)).$$
By the definition of $j_0$ and the propagation property of Theorem \ref{package}, we have
\begin{equation}\label{general1}
V\not\subset \sV^i(G, \,^p\sH^j(\sF))
\end{equation}
for any $i$ and $j>j_0$. Furthermore, by the abelian codimension bound in Theorem \ref{characterization}, 
\be\label{general2}
V\not \subset \sV^i(G, \,^p\sH^j(\sF))
\ee
for any $j$ and $i>k$. Now, let $L$ be a $\bC$-local system on $G$ corresponding to a general point in $V$. By (\ref{general1}) and (\ref{general2}), if $j>j_0$ or $i>k$, then
$$H^i(G, ^p\sH^j(\sF)\otimes_\bC L)=0.$$
Therefore, the above spectral sequence satisfies
$$H^k(G, ^p\sH^{j_0}(\sF)\otimes_\bC L)=E_2^{k, j_0}=E_\infty^{k, j_0}\neq 0.$$
Thus, $V\subset\sV^{k+j_0}(G,\sF)\subset\bigcup_i \sV^i(G,\sF)$, and hence
\[  \bigcup_j \sV^0(G, \,^p\sH^j(\sF)) \subset \bigcup_i \sV^i(G,\sF).\qedhere
\]
\end{proof}

\section{Applications}\label{appl}
In this section, we present applications of our main results to the cohomology jump loci of quasi-projective manifolds, to the topology of the Albanese map, and to the study of abelian duality spaces.

\subsection{Cohomology jump loci of quasi-projective manifold}\label{cjl}
In this subsection, we give some applications of Theorem \ref{package} and Theorem \ref{characterization} to the study of cohomology jump loci of smooth complex quasi-projective varieties. 

Let $X$ be a smooth connected complex quasi-projective variety. The {\it character variety} $\Char(X)$ is the connected component of $\Hom(\pi_1(X),\bC^*)$ containing the identity. $\Char(X)$ is isomorphic to
$(\bC^*)^{b_1(X)}$, and it is identified with the maximal spectrum $\spec \bC[H_{1,f}(X,\bZ)]$ of the group ring $\bC[H_{1,f}(X,\bZ)]$ of the free part of $H_1(X,\bZ)$.

Let  $\alb: X \to \Alb(X)$ be the Albanese map associated to $X$ (see e.g. \cite{Iit}). The Albanese variety $\Alb(X)$ is a semi-abelian variety and the Albanese map $\alb$ induces an isomorphism between the free abelian part of $H_1(X,\bZ)$ and $H_1(\Alb(X),\bZ)$. 
Therefore, $\Char(X) \cong \Char(\Alb(X))$.

\bd 
The {\it cohomology jump loci of $X$} are defined as:
\be\label{cre} \sV^{i}(X): =  \{\rho \in \Char(X) \mid   H^i(X,  L_\rho) \neq 0\} ,\ee 
 where $L_\rho$ is as before the rank-one $\bC$-local system on $X$ associated to $\rho$. 
\ed
Note that $ \sV^0(X)= \{\mathbf{1} \}$, where $\mathbf{1}$ denotes the trivial character.

\bc \label{manifold}
Let  $X$ be a  smooth quasi-projective variety of complex dimension $n$.  Assume that $ \text{ }^p \sH^j (R\alb_* \bC_X[n])=0$ for $j\notin [b,c]$. Then the cohomology jump loci $\sV^i(X)$ have the following properties:
\begin{itemize}
\item[(1)] {\it Propagation property}:
$$ 
  \sV^{n+b}(X) \supseteq \sV^{n+b-1}(X) \supseteq \cdots \supseteq \sV^{0}(X) = \{ \mathbf{1} \};
$$
$$ 
  \sV^{n+c}(X)  \supseteq \sV^{n+c+1}(X) \supseteq \cdots \supseteq \sV^{2n}(X)  .
$$
\item[(2)] {\it Codimension lower bound}: for any $ i\geq 0$, $$ \codim_{sa} \sV^{n+b-i}(X)\geq i \ \ {\rm and} \  \codim_{a} \sV^{n+c+i}(X) \geq i.$$
\item[(3)]For generic $\rho \in  \Char(X)$ and all $ i \notin [n+b, n+c]$, $$H^{i}(X,  L_\rho)=0.$$ 
\item[(4)] 
$b_i(X)>0$ for any $i\in [0, n+b]$,
and 
$b_1(X)\geq n+b$.
\end{itemize} 
\ec

\begin{proof} 
The second and third claims follow directly from Theorem \ref{characterization}. 
The first part of the propagation property follows from Corollary \ref{mell} and Remark \ref{half}.    The second part of the propagation property follows from an analogous dual argument.

The fact that $b_i(X)>0$ for any $i\in [0, n+b]$ follows from the propagation property, since $\{\mathbf{1}\} \in\sV^{ 0} (X)$. The codimension lower bound yields that $\codim_{sa} \sV^0(X)= b_1(X)\geq n+b$. Therefore, $b_1(X)\geq k \geq n+b$.
\end{proof}

\br\label{ssm} 
 Assume that the Albanese map $\alb :X\to \Alb(X)$ is proper. Let $$r(\alb)=\dim (X \times_{\Alb(X)} X) -\dim X$$ be the \emph{defect of semi-smallness} of $\alb$. Then, by the decomposition theorem \cite{BBD}, in Corollary \ref{manifold} we have that  $[b,c]=[-r(\alb),r(\alb)].$  

One particular interesting case is when $\alb$ is proper and semi-small, in which case $r(\alb)=0$. It was shown in \cite[Remark 1.3]{LMW2} that  $X$ admits a proper semi-small map   $f : X \to G$ to some complex semi-abelian variety $G$ if and only if   the Albanese map $\alb: X \to \Alb(X)$ is proper and semi-small.   It is sometimes easier to construct a proper semi-small map $f$ to a complex semi-abelian variety than to check directly if $\alb$ is proper and semi-small, e.g., see Examples \ref{va} and \ref{ample} below.
\er


\subsection{Topology of the Albanese map}\label{am}

In this subsection, we give some applications of Theorem \ref{package} and Theorem \ref{simple} to the topological study of the Albanese map $\alb : X \to \Alb(X)$ corresponding to a smooth complex quasi-projective variety $X$. 

\bc\label{isol}  Let $X$ be an $n$-dimensional smooth complex quasi-projective variety.  
If $\bigcup_{i=0}^{2n} \sV^i(X)$ contains an isolated point, then $\alb : X \to \Alb(X)$ is dominant. 
\ec
\begin{proof}

Set $\sF:= R\alb_* \bC_X[n]$, which is a bounded $\bC$-constructible complex on $\Alb(X)$. By the projection formula, we have 
$$\bigcup_{i} \sV^i(X) =\bigcup_i \sV^i(\Alb(X),\sF).$$
Furthermore, Proposition \ref{Alexander} shows that 
$$  \bigcup_i \sV^i(\Alb(X),\sF) =\bigcup_j \sV^0(\Alb(X), \,^p\sH^j(\sF)).$$
Since, by our assumption,  $\bigcup_{i=0}^{2n} \sV^i(X)$ contains an isolated point, this isolated point is contained in $\sV^0(\Alb(X), \,^p\sH^{j_0}(\sF)) $, for some integer $j_0$. 

 Corollary \ref{union} yields that there exists at least one simple perverse sheaf $\sP$ on $\Alb(X)$ such that $\sP$ is a decomposition factor of $ \,^p\sH^{j_0}(\sF)$ and $\sV^0(\Alb(X),\sP)$ is exactly this isolated point. Then it follows from Theorem \ref{simple} that $\sP$ is a rank-one $\bC$-local system on $\Alb(X)$.  So $\alb$ is dominant. 
\end{proof}

\bc\label{fmp} Let $X$ be an $n$-dimensional smooth complex quasi-projective variety with Albanese map $\alb : X \to \Alb(X)$. If $\bigcup_{i=0}^{2n} \sV^i(X)$ consists of finitely many points, then $R^i\alb_* \bC_X$ is a  $\bC$-local system on $\Alb(X)$ for every $i$. If, moreover, $alb$ is proper, then  $R\alb_* \bC_X[n]$ is a direct sum of shifted rank-one $\bC$-local system on $\Alb(X)$.
\ec
\begin{proof} 
Set $\sF:= R\alb_* \bC_X[n]$. As in the proof of Corollary \ref{isol},  each $\sV^0(\Alb(X), \,^p\sH^j(\sF))$ is either empty or consists of finitely many points. Since every perverse sheaf is the extension of finitely many simple perverse sheaves, we have by Theorem \ref{simple} and Corollary \ref{union} that $^p\sH^j(\sF)$ is a shift of a local system on $\Alb(X)$ for every $j$. Thus, $R^i\alb_* \bC_X=\sH^{i-n}(\sF)$ is a $\bC$-local system on $\Alb(X)$ for every $i$ (see, e.g., \cite[Proposition~4.3]{LMW20}).

For the last assertion, we use the decomposition theorem \cite{BBD} which yields that $R\alb_* \bC_X[n]$ is a direct sum of (shifted) simple perverse sheaves. As above, Theorem \ref{simple} implies that each of these simple perverse sheaves is (up to a shift) a rank-one $\bC$-local system on $\Alb(X)$.
\end{proof}

\bc \label{isomorphic}  Let $X$ be a connected smooth complex quasi-projective variety of dimension $n$, with proper and semi-small Albanese map $\alb : X \to \Alb(X)$. Then $\sV^n(X)$ consists of finitely many points if, and only if, $\alb$ is an isomorphism.
\ec
\begin{proof} 
The ``if" part is obvious. We prove the ``only if" part.
Since $\alb : X \to \Alb(X)$ is proper and semi-small, $R \alb_* \bC_X[n]$ is a perverse sheaf on $\Alb(X)$. By  Corollary \ref{manifold}(1) we have
$$\bigcup_{i=0}^{2n} \sV^i(X) = \sV^n(X).$$ 
If $\sV^n(X)$ consists of finitely many points, then by Corollary \ref{fmp} we have that $R^i \alb_* \bC_X$ is a finite direct sum of rank-one $\bC$-local systems for every $i$. Since $R\alb_* \bC_X[n]$ is a perverse sheaf on $\Alb(X)$, $R^i \alb_* \bC_X=0$ unless $i=0$. 
 Furthermore, Corollary \ref{survive} yields that $\sV^0(X)=\sV^n(X)$. But $\sV^0(X)$ consists of only one point, the constant sheaf. This implies that $R\alb_* \bC_X =\bC^{\oplus k}_{\Alb(X)}$, a finite direct sum of the constant sheaf on $\Alb(X)$.  Since $X$ is connected, 
\[
\bC=H^0(X, \bC_X)= \mathbb{H}^0(\Alb(X), R\alb_* \bC_X)=\bC^k,
\]
which means that $k=1$. Since $R\alb_* \bC_X =\bC^{\oplus k}_{\Alb(X)}$, all fibers of $\alb$ are zero-dimensional. In other words, $\alb$ is quasi-finite. A proper quasi-finite map is finite. Since $k=1$, the Albanese map $\alb: X\to \Alb(X)$ is a finite morphism of degree one, that is, an isomorphism. 
\end{proof}

The following generalization of \cite[Corollary 2.6]{BC} gives a topological characterization of semi-abelian varieties.
\bp\label{topsa} Let $X$ be a smooth quasi-projective variety with proper Albanese map (e.g., $X$ is projective), and assume that $X$ is homotopy equivalent to a torus. Then $X$ is isomorphic to a semi-abelian variety.
\ep

\begin{proof}
Since $X$ is homotopy equivalent to a torus, $\bigcup_{i\geq  0} \sV^i(X)=\{ \mathbf{1}\}$. By Corollary \ref{fmp}, $R\alb_* \bC_X$ is a direct sum of shifted rank-one constant sheaves on $\Alb(X)$. Since $X$ and $\Alb(X)$ are both homotopy equivalent to tori, and since $b_1(X)=b_1(\Alb(X))$, we have that $b_i(X)=b_i(\Alb(X))$ for any $i$. Therefore, $R\alb_* \bC_X\cong \bC_{\Alb(X)}$. Now, the same argument as in Corollary \ref{isomorphic} shows that $\alb: X\to \Alb(X)$ is an isomorphism. 
\end{proof}


\subsection{Abelian duality spaces}\label{ads}

Let us recall the definition of (partial) abelian duality spaces from \cite{LMW3}, see also \cite{DSY}.

Let $X$ be a connected finite CW complex, and denote $\pi_1(X)$ by $\pi$. Let $\phi: \pi\to \pi'$ be a non-trivial homomorphism to an abelian group $\pi'$. There is a canonical $\bZ [\pi^\prime]$-local coefficient system $\sL_{\phi}$ on $X$, whose monodromy action is given by the composition of $\pi\overset{\phi}{ \to} \pi^\prime$ with the natural multiplication  $\pi^\prime\times \bZ[ \pi^\prime] \to \bZ [\pi^\prime]$.  
\bd\label{dpad} We call $X$ a {\it partially abelian duality space of dimension $n$ with respect to} $\phi:\pi \to \pi'$, if the following two conditions are satisfied:
\begin{itemize}
\item[(a)] $H^i(X, \bZ [\pi'])=0$ for $i\neq n$, 
\item[(b)] $H^n(X, \bZ [\pi'])$ is a (non-zero) torsion-free $\bZ$-module. 
\end{itemize}
If $\pi^\prime=\pi^{ab}=H_1(X, \bZ)$ and $\phi$ is the abelianization map, then a finite connected CW complex $X$ satisfying (a) and (b) is called an {\it abelian duality space of dimension $n$}, see \cite{DSY}.
\ed

\br There is a canonical $\bZ[\pi']$-module isomorphism
$$H^i (X, \bZ[\pi']) \cong H^i(X,\sL_{\phi}),$$
for any $i$.
\er

Examples of (partially) abelian duality spaces where constructed in \cite[Theorem 4.11]{LMW3} via algebraic maps to complex affine tori. 
The following  two results provide generalizations of \cite[Theorem 4.11]{LMW3} to the semi-abelian setting.
\bt\label{sab} Let $X$ be an $n$-dimensional smooth complex  quasi-projective variety, and let $f:X \to G$ be an algebraic map to a semi-abelian variety $G$. Assume that $X$ is homotopy equivalent to an $n$-dimensional CW complex (e.g., $X$ is affine). If $Rf_*\kk_X[n] \in \,^p D^{\geq 0}(G,\kk)$ for any field $\kk$ (e.g., if $f$ is quasi-finite), then $X$ is a partially abelian duality space of dimension $n$ with respect to $f_*:\pi_1(X) \to \pi_1(G)$. 
\et

\begin{proof} We first show that for any field $\kk$, the value of the corresponding Mellin transformation on $\sF:=Rf_*\kk_X[n]$ has nonzero cohomology only in degree zero.

First, by Corollary \ref{mell}, we have that 
\be\label{les} H^i(\sM_*(\sF))=0, \  {\rm for } \ i<0.\ee
Secondly, 
$$H^i(\sM_*(\sF))\cong H^{i+n}(G,\sL_G \otimes_{\kk}Rf_*\kk_X)\cong H^{i+n}(X,f^*\sL_G),$$
where the last isomorphism follows by the projection formula (since $\sL_G$ is a local system). Finally, since $X$ has the homotopy type of an $n$-dimesional CW-complex, and 
$f^*\sL_G$ is a local system on $X$, we have that 
$H^{i+n}(X,f^*\sL_G)=0$ for $i>0$. Hence, \be\label{ges} H^i(\sM_*(\sF))=0, \  {\rm for } \ i>0.\ee

Altogether, $$H^i(\sM_*(\sF))=0, \  {\rm for } \ i\neq 0.$$
The desired result follows now by using the same argument as in \cite[Theorem 4.11]{LMW3}(1).
\end{proof}

Let us now specialize to the case when $G=\Alb(X)$ is the Albanese variety of $X$, and $f=\alb$ is the Albanese map. 

\bt \label{abelianduality} Let $X$ be an $n$-dimensional smooth complex quasi-projective variety, which is homotopy equivalent to an $n$-dimensional CW complex (e.g., $X$ is affine).  Suppose the Albanese map $\alb$ is proper and semi-small, or $\alb$ is quasi-finite. Then $X$ is an abelian duality space of dimension $n$. 
\et
\begin{proof}
The assumptions on $\alb$ imply that $R\alb_* L[n]$ is a perverse sheaf for any local system $L$ over any field $\kk$. By the arguments in the proof of \cite[Theorem 4.11]{LMW3}, it suffices to show that $\sM_*(R\alb_* L[n])$ has nonzero cohomology only in degree zero. The assertion follows exactly as in the proof of Theorem \ref{sab}. 
\end{proof}

\bex\label{va}
Let $X$ be an $n$-dimensional {\it very affine manifold}, i.e., a smooth closed subvariety of a complex affine torus $T=(\bC^*)^m$ (e.g., the complement of an essential hyperplane arrangement or of a toric arrangement). The closed embedding  $i:X\hookrightarrow T$ is a proper semi-small map, and hence $\alb:X \to \Alb(X)$ is also proper and semi-small. Since $X$ is also affine, we get by Theorem \ref{abelianduality} that $X$ is an abelian duality space of dimension $n$. (This example generalizes \cite[Example 5.1]{LMW3}.)
\eex

\bex \label{ample}  Let $Y$ be a smooth complex projective variety, and let $\sL$ be a very ample line bundle on $Y$. Consider an $N$-dimensional sub-linear system $\vert E\vert $ of $\vert \sL \vert$ such that $E$ is base point free over $Y$. Then a basis $\{s_0, s_1, \cdots, s_N\}$ of $E$ gives  a well-defined morphism  $$\varphi_{\vert E \vert}:   Y \to \mathbb{CP}^N .$$ Each  $\{ s_i=0 \} $ defines a  hypersurface $V_i$ in $Y$. In particular, $\bigcap_{i=0}^N V_i = \emptyset$. Then $\varphi_{\vert E \vert}$ is a finite morphism; for a proof, see \cite[Example 5.4]{LMW3}.
Taking the restriction $f$ of $\varphi_{\vert E \vert}$ over  $X= Y\setminus \bigcup_{i=0}^N V_i $, we get a map
\begin{align*}
 f:=\left(\dfrac{s_1}{s_0}, \dfrac{s_2}{s_0}, \cdots,\dfrac{s_N}{s_0}\right): X 
 \lra T=(\bC^*)^N \end{align*}
which is finite, hence  proper and semi-small. As discussed in Remark \ref{ssm}, this implies that the albanese map $\alb$ is also proper and semi-small.
Theorem \ref{abelianduality} yields that $X$ is an abelian duality space. 
\eex

\br
Example \ref{ample} above extends \cite[Example 5.4]{LMW3}, where the first Betti number of $Y$ was required to vanish. If all hypersurfaces $V_i$ are smooth and they intersect locally like hyperplanes, then the above example is a special case of \cite[Theorem 1.1]{DS}. However, we do not need any assumption on the singularities of the $V_i$'s and their intersections. 
\er

\bex  It is shown in \cite{DS,DSY} that the complement of an {\it elliptic arrangement} is an abelian duality space. This fact also follows from Theorem \ref{abelianduality} as we shall now indicate. Let $E$ be an elliptic curve, and let $\mathcal{A}$ be an essential elliptic arrangement in $E^n$ with complement $X:=E^n \setminus \mathcal{A}$. Then $X$ is a complex $n$-dimensional affine variety. By the universal property of the Albanese map, the natural embedding $X \hookrightarrow E^n$ factorizes through $\alb:X \to \Alb(X)$. Hence the Albanese map $\alb:X \to \Alb(X)$ is also an embedding (hence, in particular, quasi-finite). So Theorem \ref{abelianduality} applies to show that $X$ is an abelian duality space of dimension $n$. \eex

The affine condition is a sufficient but not a necessary condition for an $n$-dimensional smooth complex quasi-projective variety to be homotopy equivalent to a finite CW-complex of dimension $n$. The following is a simple example of an $n$-dimensional smooth complex quasi-projective variety which is an abelian duality space of dimension $n$, but which is not affine. 
\bex
Let $X$ be the blowup of $(\cc^*)^2$ at a point. Then $X$ is an abelian duality space of dimension 2. Indeed, the Albanese map is the blowdown map $X\to (\cc^*)^2$, which is proper and semismall. 
Moreover, $X$ is homotopy equivalent to the $2$-dimensional CW-complex $T^2 \vee S^2$, where $T^2=S^1\times S^1$ is the real $2$-dimensional torus. Thus, $X$ is an abelian duality space by Theorem \ref{abelianduality}. However, $X$  is not affine because it contains a closed subvariety $\mathbb{CP}^1$. 
\eex

If an $n$-dimensional smooth complex quasi-projective variety $X$ is an abelian duality space and $\alb: X\to \Alb(X)$ is proper, then $\alb$ is semi-small. Indeed, by using the decomposition theorem and the relative hard Lefschetz theorem, one can readily see that $R\alb_* \bC_X[n]$ is a perverse sheaf.  We conjecture that, in some sense, the converse of Theorem \ref{abelianduality} is also true. 
\begin{conj} Let $X$ be an $n$-dimensional smooth complex quasi-projective variety with proper Albanese map $\alb : X \to \Alb(X)$. Then $X$ is an abelian duality space of dimension $n$ if and only if  $\alb$ is semi-small and $X $ is homotopy equivalent to a finite $n$-dimensional CW complex.  \end{conj}


\end{document}